\documentclass[12pt]{amsart}
\usepackage[cmtip,all]{xy}
\usepackage{graphicx}
\usepackage{amssymb}
\usepackage{mathrsfs}
\usepackage{amsfonts}
\usepackage{amsmath}

\newtheorem{theorem}{Theorem}[section]
\newtheorem{lemma}[theorem]{Lemma}

\theoremstyle{definition}

 \theoremstyle{remark}
\newtheorem{remark}[theorem]{Remark}
\newtheorem{corollary}[theorem]{Corollary}

 \numberwithin{equation}{section}

\begin{document}

\title[High Order Hardy-Sobolev-Maz'ya inequalities]{Hardy-Sobolev-Maz'ya inequalities for higher order derivatives on half spaces}

\author{Guozhen Lu}
\address{Department of Mathematics,  University of Connecticut, Storrs, CT 06269, USA}

\email{guozhen.lu@uconn.edu}

\author{Qiaohua  Yang}
\address{School of Mathematics and Statistics, Wuhan University, Wuhan, 430072, People's Republic of China}

\email{qhyang.math@gmail.com}

\thanks{ The first author was partly supported by a  US NSF grant and a Simons Fellowship from the Simons Foundation and the second author's research was partly supported by   the National Natural
Science Foundation of China (No.11201346).}






\begin{abstract}

By using, among other things,  the Fourier analysis techniques on hyperbolic and symmetric spaces, 
we establish the Hardy-Sobolev-Maz'ya inequalities for higher order derivatives on half spaces. The proof relies on a Hardy-Littlewood-Sobolev inequality on hyperbolic spaces which is of its independent interest.
We also give an alternative proof of  Benguria,  Frank and   Loss' work  concerning the sharp constant in the Hardy-Sobolev-Maz'ya inequality in the three dimensional upper half space. Finally, we show the sharp constant in the Hardy-Sobolev-Maz'ya inequality for bi-Laplacian in the   upper half space of dimension five coincides with the Sobolev constant.

\end{abstract}

\maketitle


\section{Introduction}
Let $n\geq3$. The Hardy-Sobolev-Maz'ya inequalities on half space $\mathbb{R}^{n}_{+}=\{(x_{1},\cdots,x_{n}): x_{1}>0\}$ reads that (see \cite{maz}, Section 2.1.6)
\begin{equation}\label{1.1}
\int_{\mathbb{R}^{n}_{+}}|\nabla u|^{2}dx-\frac{1}{4}\int_{\mathbb{R}^{n}_{+}}\frac{u^{2}}{x^{2}_{1}}dx\geq C\left(\int_{\mathbb{R}^{n}_{+}}
x^{\gamma}_{1}|u|^{p}dx\right)^{\frac{2}{p}},\;\;\; u\in C^{\infty}_{0}(\mathbb{R}^{n}_{+}),
\end{equation}
where $C$ is a positive constant which is independent of $u$,  $2<p\leq\frac{2n}{n-2}$ and $\gamma=\frac{(n-2)p}{2}-n$.
In particular, $\gamma=0$ for $p=\frac{2n}{n-2}$ and we have
\begin{equation}\label{1.2}
\int_{\mathbb{R}^{n}_{+}}|\nabla u|^{2}dx-\frac{1}{4}\int_{\mathbb{R}^{n}_{+}}\frac{u^{2}}{x^{2}_{1}}dx\geq C_{n}\left(\int_{\mathbb{R}^{n}_{+}}
|u|^{\frac{2n}{n-2}}dx\right)^{\frac{n-2}{2}},\;\;\; u\in C^{\infty}_{0}(\mathbb{R}^{n}_{+}).
\end{equation}
It has been shown by R. D. Benguria, R. L. Frank and  M. Loss (\cite{bfl}) that the sharp constant $C_{3}$ in (\ref{1.2}) for $n=3$ coincides with
the corresponding Sobolev constant.
In the paper \cite{m1}, G. Mancini and K. Sandeep showed inequality (\ref{1.1}) is equivalent to the following
Poincar\'e-Sobolev inequalities on hyperbolic space $\mathbb{H}^{n} \, (n\geq3)$:
\begin{equation}\label{1.3}
\int_{\mathbb{H}^{n}}|\nabla_{\mathbb{H}}u|^{2}dV-\frac{(n-1)^{2}}{4}\int_{\mathbb{H}^{n}}u^{2}dV\geq C\left(\int_{\mathbb{H}^{n}}|u|^{p}dV\right)^{\frac{2}{p}},\;\;u\in C^{\infty}_{0}(\mathbb{H}^{n}),
\end{equation}
where $2<p\leq\frac{2n}{n-2}$, $\nabla_{\mathbb{H}}$ is the hyperbolic gradient and  $dV$ is the hyperbolic volume element.
For $n=2$, there holds some Hardy-Trudinger-Moser inequality on $\mathbb{H}^{2}$ (see \cite{mst,ly, wy}).
For more information about Hardy-Sobolev-Maz'ya inequalities, we refer to \cite{fmt1,fmt2,fl,mas,te}.

\medskip

A natural question is whether  inequalities (\ref{1.1}) and (\ref{1.3}) hold for higher  order
derivatives. In this paper we shall show this is indeed the case. To state our results, let us introduce  some conventions.
It is known that hyperbolic space has its half space model and Poincar\'e model and both models are equivalent.
We denote by $\mathbb{B}^{n}$  the Poincar\'e model. It is the unit ball
\[\mathbb{B}^{n}=\{x=(x_{1},\cdots,x_{n})\in \mathbb{R}^{n}| |x|<1\}\]
equipped with the usual Poincar\'e metric
\[
ds^{2}=\frac{4(dx^{2}_{1}+\cdots+dx^{2}_{n})}{(1-|x|^{2})^{2}}.
\]
The hyperbolic volume element is
$
dV=\left(\frac{2}{1-|x|^{2}}\right)^{n}dx$ and the distance from the origin to $x\in \mathbb{B}^{n}$ is
$
\rho(x)=\log\frac{1+|x|}{1-|x|}.$ The associated
Laplace-Beltrami operator is given by
\[
\Delta_{\mathbb{H}}=\frac{1-|x|^{2}}{4}\left\{(1-|x|^{2})\sum^{n}_{i=1}\frac{\partial^{2}}{\partial
x^{2}_{i}}+2(n-2)\sum^{n}_{i=1}x_{i}\frac{\partial}{\partial
x_{i}}\right\}.
\]
The  spectral gap of $-\Delta_{\mathbb{H}}$ on
$L^{2}(\mathbb{B}^{n})$ is $\frac{(n-1)^{2}}{4}$ (see e.g.
\cite{m1}), i.e.
\begin{equation}\label{1.4}
\int_{\mathbb{B}^{n}}|\nabla_{\mathbb{H}}u|^{2}dV\geq\frac{(n-1)^{2}}{4}\int_{\mathbb{B}^{n}}u^{2}dV,\;\;\;
u\in C^{\infty}_{0}(\mathbb{B}^{n}).
\end{equation}
The the GJMS operators on $\mathbb{B}^{n}$  is defined as follows (see \cite{GJMS},  \cite{j})
 \begin{equation}\label{1.5}
 P_{k}=P_{1}(P_{1}+2)\cdot\cdots\cdot(P_{1}+k(k-1)),\;\; k\in\mathbb{N},
\end{equation}
where $P_{1}=-\Delta_{\mathbb{H}}-\frac{n(n-2)}{4}$ is the conformal Laplacian on $\mathbb{B}^{n}$.
The sharp Sobolev inequalities on $\mathbb{B}^{n}$  reads that (see \cite{heb} for $k=1$ and \cite{liu} for $2\leq k<\frac{n}{2}$)
 \begin{equation}\label{1.6}
\int_{\mathbb{B}^{n}}(P_{k}u)udV\geq S_{n,k}\left(\int_{\mathbb{B}^{n}}|u|^{\frac{2n}{n-2k}}dV\right)^{\frac{n-2k}{n}},\;\;u\in
C^{\infty}_{0}(\mathbb{B}^{n}),\;\;1\leq k<\frac{n}{2},
\end{equation}
where $S_{n,k}$ is the best $k$-th order Sobolev constant. On the other hand,
By (\ref{1.4}) and (\ref{1.5}), we have the following Poincar\'e inequality
 \begin{equation}\label{1.7}
\int_{\mathbb{B}^{n}}(P_{k}u)udV\geq \prod^{k}_{i=1}\frac{(2i-1)^{2}}{4}\int_{\mathbb{B}^{n}}u^{2}dV,\;\;u\in C^{\infty}_{0}(\mathbb{B}^{n}).
\end{equation}
As we will show in the proof of Theorem 1.4 (see Section 5), inequality (\ref{1.7}) is equivalent to the Hardy inequality on the upper half space
\[
\int_{\mathbb{R}^{n}_{+}}|\nabla^{k}u|^{2}dx\geq \prod^{k}_{i=1}\frac{(2i-1)^{2}}{4}\int_{\mathbb{R}^{n}_{+}}\frac{u^{2}}{x^{2k}_{1}}dx,\;\; u\in C^{\infty}_{0}(\mathbb{R}^{n}_{+}),
\]
and the constant $\prod\limits^{k}_{i=1}\frac{(2i-1)^{2}}{4}$ is sharp (see \cite{ow}).

\medskip

Next we define another $2k$-th order operator $Q_{k}$ with $k\geq2$:
 \begin{equation}\label{1.8}
\begin{split}
Q_{k}=&\left(-\Delta_{\mathbb{H}}-\frac{(n-1)^{2}}{4}\right)(P_{1}+2)\cdot\cdots\cdot(P_{1}+k(k-1))\\
=&P_{k}-\frac{1}{4}(P_{1}+2)
\cdot\cdots\cdot(P_{1}+k(k-1)).
\end{split}
 \end{equation}
To this end, we have the following Sobolev type inequalities for $Q_{k}$.
\begin{theorem}
Let $2\leq k<\frac{n}{2}$ and $2<p\leq\frac{2n}{n-2k}$. There exists a positive constant $C$ such that for each $u\in C^{\infty}_{0}(\mathbb{B}^{n})$,
\begin{equation}\label{1.9}
\int_{\mathbb{B}^{n}}(Q_{k}u)udV\geq C\left(\int_{\mathbb{B}^{n}}|u|^{p}dV\right)^{\frac{2}{p}}.
\end{equation}
\end{theorem}

The proof of Theorem 1.1 depends on a Hardy-Littlewood-Sobolev inequlity
on hyperbolic spaces which is of independent interest.
\begin{theorem}
Let $0<\lambda<n$ and $p=\frac{2n}{2n-\lambda}$. Then for $f,g\in L^{p}(\mathbb{B}^{n})$,
\begin{equation}\label{b1.10}
\left|\int_{\mathbb{B}^{n}}\int_{\mathbb{B}^{n}}\frac{f(x)g(y)}{\left(2\sinh\frac{\rho(T_{y}(x))}{2}\right)^{\lambda}}dV_{x}dV_{y}\right|\leq C_{n,\lambda}\|f\|_{p}\|g\|_{p},
\end{equation}
 where $T_{y}(x)$ is the M\"obius
transformations (see Section 2.2) and
\begin{equation}\label{1.10}
C_{n,\lambda}=\pi^{\lambda/2}\frac{\Gamma(n/2-\lambda/2)}{\Gamma(n-\lambda/2)}\left(\frac{\Gamma(n/2)}{\Gamma(n)}\right)^{-1+\lambda/n}
\end{equation}
 is the best Hardy-Littlewood-Sobolev constant on $\mathbb{R}^{n}$.
Furthermore, the constant $C_{n,\lambda}$
 is sharp and there is no nonzero extremal function.
\end{theorem}

The reader may  wonder why function $\sinh\frac{\rho}{2}$ appears in the inequality (\ref{b1.10}). In fact, the Green's function of conformal Laplacian
$-\Delta_{\mathbb{H}}-\frac{n(n-2)}{4}$ is (see Section 3, (\ref{3.5}))
\begin{equation*}
\begin{split}
(-\Delta_{\mathbb{H}}-n(n-2)/4)^{-1}
=&\frac{\Gamma(\frac{n}{2})}{2(2\pi)^{\frac{n}{2}}(\sinh\rho)^{n-2}}\int^{\pi}_{0}(\cosh\rho+\cos t)^{\frac{n-4}{2}}
\sin tdt\\
=&\frac{\Gamma(\frac{n}{2})}{2(2\pi)^{\frac{n}{2}}(\sinh\rho)^{n-2}}\cdot\frac{2}{n-2}\left[(\cosh\rho+1)^{\frac{n-2}{2}}-(\cosh\rho-1)^{\frac{n-2}{2}}\right]\\
=&\frac{1}{n(n-2)\alpha(n)}\left[\frac{1}{(2\sinh\frac{\rho}{2})^{n-2}}-\frac{1}{(2\cosh\frac{\rho}{2})^{n-2}}\right], \;\;n\geq3,
\end{split}
\end{equation*}
where $\alpha(n)=\frac{\pi^{\frac{n}{2}}}{\Gamma(\frac{n}{2}+1)}$ denotes the volume of the unit ball in $\mathbb{R}^{n}$.

\medskip

Notice that, by (\ref{1.4}),
 \begin{equation}\label{b1.12}
\int_{\mathbb{B}^{n}}((P_{1}+2)
\cdot\cdots\cdot(P_{1}+k(k-1))u)udV\geq \prod^{k}_{i=2}\frac{(2i-1)^{2}}{4}\int_{\mathbb{B}^{n}}u^{2}dV,\;\;u\in C^{\infty}_{0}(\mathbb{B}^{n}).
\end{equation}
Combing (\ref{b1.12}), (\ref{1.7}) and (\ref{1.8}) yields
\begin{equation*}
\int_{\mathbb{B}^{n}}(P_{k}u)udV- \prod^{k}_{i=1}\frac{(2i-1)^{2}}{4}\int_{\mathbb{B}^{n}}u^{2}dV\geq \int_{\mathbb{B}^{n}}(Q_{k}u)udV.
\end{equation*}
Therefore, as an application of Theorem 1.1, we have the following Poincar\'e-Sobolev inequalities for higher order derivatives.
\begin{theorem}
Let $2\leq k<\frac{n}{2}$ and $2<p\leq\frac{2n}{n-2k}$. There exists a positive constant $C_{n,p}$ such that for each $u\in C^{\infty}_{0}(\mathbb{B}^{n})$,
\begin{equation}\label{1.11a}
\int_{\mathbb{B}^{n}}(P_{k}u)udV- \prod^{k}_{i=1}\frac{(2i-1)^{2}}{4}\int_{\mathbb{B}^{n}}u^{2}dV\geq C_{n,p}\left(\int_{\mathbb{B}^{n}}|u|^{p}dV\right)^{\frac{2}{p}}.
\end{equation}
\end{theorem}

This improves the  Poincar\'e-Sobolev inequalities  (\ref{1.6}) for higher order derivatives on the hyperbolic spaces $\mathbb{B}^n$
established  by G. Liu \cite{liu}.

If $p=\frac{2n}{n-2k}$, then, by (\ref{1.6}), the sharp constant $C_{n,p}$ in (\ref{1.11a}) is less than or equal to the the best $k$-th order Sobolev constant.

As an application of Theorem 1.3, we have the following Hardy-Sobolev-Maz'ya inequalities for higher order derivatives.
\begin{theorem}
Let $2\leq k<\frac{n}{2}$ and $2<p\leq\frac{2n}{n-2k}$. There exists a positive constant $C$ such that for each $u\in C^{\infty}_{0}(\mathbb{R}^{n}_{+})$,
\begin{equation}\label{1.12}
\int_{\mathbb{R}^{n}_{+}}|\nabla^{k}u|^{2}dx- \prod^{k}_{i=1}\frac{(2i-1)^{2}}{4}\int_{\mathbb{R}^{n}_{+}}\frac{u^{2}}{x^{2k}_{1}}dx\geq C\left(\int_{\mathbb{R}^{n}_{+}}x^{\gamma}_{1}|u|^{p}dx\right)^{\frac{2}{p}},
\end{equation}
where $\gamma=\frac{(n-2k)p}{2}-n$.

\medskip

In term of the ball model $\mathbb{B}^{n}$ (see section 2.2), inequality (\ref{1.12}) can be written as follows:
\begin{equation}\label{1.14}
\int_{\mathbb{B}^{n}}|\nabla^{k}u|^{2}dx- \prod^{k}_{i=1}(2i-1)^{2}\int_{\mathbb{B}^{n}}\frac{u^{2}}{(1-|x|^{2})^{2k}}dx\geq C\left(\int_{\mathbb{B}^{n}}(1-|x|^{2})^{\gamma}|u|^{p}dx\right)^{\frac{2}{p}}.
\end{equation}
\end{theorem}

 The organization of this paper is as follows. In Section 2, we recall some necessary preliminary facts of hyperbolic spaces and give 
 the sharp estimates of Green's function  in Section 3. we prove Theorem 1.1 and Theorem 1.2. in   Section 4.  Theorem 1.4 is proved in Section 5.  In Section 6, We give an alternative proof of  the work of Benguria, Frank and   Loss \cite{bfl} concerning the sharp constant in the Hardy-Sobolev-Maz'ya inequality in case $n=3$.
In Section 7, we show that, as in the case of Euclidean space, the Hardy-Littlewood-Sobolev inequlity on hyperbolic space implies the sharp Sobolev inequalities obtained by Liu \cite{liu}. In the last section, we prove the best constant  coincides with the Sobolev constant in case $n=5$ and $k=2$, also via Hardy-Littlewood-Sobolev inequlity.

\section{Notations and preliminaries}
We begin by quoting some preliminary facts which will be needed in
the sequel and  refer to \cite{ah,ge,he,he2,hu,liup} for more information about this subject.

\subsection{The half space model $\mathbb{H}^{n}$}

It is given by $\mathbb{R}_{+}\times\mathbb{R}^{n-1}=\{(x_{1},\cdots,x_{n}): x_{1}>0\}$ equipped with the Riemannian metric
$ds^{2}=\frac{dx_{1}^{2}+\cdots+dx_{n}^{2}}{x^{2}_{1}}$. The induced Riemannian measure can be written as $dV=\frac{dx}{x^{n}_{1}}$, where $dx$ is the Lebesgue measure on
$\mathbb{R}^{n}$.
The hyperbolic gradient is $\nabla_{\mathbb{H}}=x_{1}\nabla$ and the
Laplace-Beltrami operator on $\mathbb{H}^{n}$ is given by
\begin{equation}\label{2.1}
\Delta_{\mathbb{H}}=x^{2}_{1}\Delta-(n-2)x_{1}\frac{\partial}{\partial x_{1}},
\end{equation}
where $\Delta=\sum^{n}_{i=1}\frac{\partial^{2}}{\partial
x^{2}_{i}}$ is the Laplace operator on $\mathbb{R}^{n}$.

\subsection{The ball model $\mathbb{B}^{n}$}
It is given by the unit ball
\[\mathbb{B}^{n}=\{x=(x_{1},\cdots,x_{n})\in \mathbb{R}^{n}| |x|<1\}\]
equipped with the usual Poincar\'e metric
\[
ds^{2}=\frac{4(dx^{2}_{1}+\cdots+dx^{2}_{n})}{(1-|x|^{2})^{2}}.
\]
The hyperbolic gradient is $\nabla_{\mathbb{H}}=\frac{1-|x|^{2}}{2}\nabla$ and
the
Laplace-Beltrami operator is given by
\[
\Delta_{\mathbb{H}}=\frac{1-|x|^{2}}{4}\left\{(1-|x|^{2})\sum^{n}_{i=1}\frac{\partial^{2}}{\partial
x^{2}_{i}}+2(n-2)\sum^{n}_{i=1}x_{i}\frac{\partial}{\partial
x_{i}}\right\}.
\]
Furthermore,  the half space model $\mathbb{H}^{n}$ and the ball model $\mathbb{B}^{n}$ are  equivalent.

\subsection{M\"obius
transformations}
For each $a\in\mathbb{B}^{n}$, we define the  M\"obius
transformations $T_{a}$ by (see e.g. \cite{ah,hu})
\[
T_{a}(x)=\frac{|x-a|^{2}a-(1-|a|^{2})(x-a)}{1-2x\cdot
a+|x|^{2}|a|^{2}},
\]
where $x\cdot a=x_{1}a_{1}+x_{2}a_{2}+\cdots +x_{n}a_{n}$ denotes
the  scalar product in $\mathbb{R}^{n}$. It is known that the measure on $\mathbb{B}^{n}$ is invariant with respect to the M\"obius
transformations.
A simple calculation shows
\begin{equation}\label{2.2}
\begin{split}
T_{a}(T_{a}(x))=&x;\\
1-|T_{a}(x)|^{2}=&\frac{(1-|a|^{2})(1-|x|^{2})}{1-2x\cdot a+|x|^{2}|a|^{2}};\\
|T_{a}(x)|=&\frac{|x-a|}{\sqrt{1-2x\cdot a+|x|^{2}|a|^{2}}};\\
\sinh\frac{\rho(T_{a}(x))}{2}=&\frac{|T_{a}(x)|}{\sqrt{1-|T_{a}(x)|^{2}}}=\frac{|x-a|}{\sqrt{(1-|a|^{2})(1-|x|^{2})}};\\
\cosh\frac{\rho(T_{a}(x))}{2}=&\frac{1}{\sqrt{1-|T_{a}(x)|^{2}}}=\frac{\sqrt{1-2x\cdot a+|x|^{2}|a|^{2}}}{\sqrt{(1-|a|^{2})(1-|x|^{2})}}.
\end{split}
\end{equation}

Using the M\"obius transformations, we can define the
distance  from $x$ to $y$ in $\mathbb{B}^{n}$ as follows
\begin{equation*}
\rho(x,y)=\rho(T_{x}(y))=\rho(T_{y}(x))=\log\frac{1+|T_{y}(x)|}{1-|T_{y}(x)|}.
\end{equation*}
Also using the M\"obius transformations,  we can define the convolution of
measurable functions $f$ and $g$ on $\mathbb{B}^{n}$ by (see e.g. \cite{liup})
\begin{equation}\label{2.3}
(f\ast g)(x)=\int_{\mathbb{B}^{n}}f(y)g(T_{x}(y))dV(y)
\end{equation}
provided this integral exists. It is easy to check that
\[
f\ast g= g\ast f.
\]
Furthermore, if $g$ is radial, i.e. $g=g(\rho)$, then (see e.g. \cite{liup}, Proposition 3.15)
\begin{equation}\label{2.4}
  (f\ast g)\ast h= f\ast (g\ast h)
\end{equation}
provided $f,g,h\in L^{1}(\mathbb{B}^{n})$

\subsection{Fourier transform on hyperbolic spaces}

Set
\[
e_{\lambda,\zeta}(x)=\left(\frac{\sqrt{1-|x|^{2}}}{|x-\zeta|}\right)^{n-1+i\lambda}, \;\; x\in \mathbb{B}^{n},\;\;\lambda\in\mathbb{R},\;\;\zeta\in\mathbb{S}^{n-1}.
\]
The Fourier transform of a function  $f$  on $\mathbb{B}^{n}$ can be defined as
\[
\widehat{f}(\lambda,\zeta)=\int_{\mathbb{B}^{n}} f(x)e_{-\lambda,\zeta}(x)dV
\]
provided this integral exists. If $g\in C^{\infty}_{0}(\mathbb{B}^{n})$ is radial, then
 $$\widehat{(f\ast g)}=\widehat{f}\cdot\widehat{g}.$$
Moreover,
the following inversion formula holds for $f\in C^{\infty}_{0}(\mathbb{B}^{n})$ (see e.g.  \cite{liup}):
\[
f(x)=D_{n}\int^{+\infty}_{-\infty}\int_{\mathbb{S}^{n-1}} \widehat{f}(\lambda,\zeta)e_{\lambda,\zeta}(x)|\mathfrak{c}(\lambda)|^{-2}d\lambda d\sigma(\varsigma),
\]
where $D_{n}=\frac{1}{2^{3-n}\pi |\mathbb{S}^{n-1}|}$ and $\mathfrak{c}(\lambda)$ is the  Harish-Chandra $\mathfrak{c}$-function given by (see e.g.  \cite{liup})
\[
\mathfrak{c}(\lambda)=\frac{2^{n-1-i\lambda}\Gamma(n/2)\Gamma(i\lambda)}{\Gamma(\frac{n-1+i\lambda}{2})\Gamma(\frac{1+i\lambda}{2})}.
\]
Similarly, there holds the Plancherel formula:
\begin{equation}\label{2.5}
\int_{\mathbb{B}^{n}}|f(x)|^{2}dV=D_{n}\int^{+\infty}_{-\infty}\int_{\mathbb{S}^{n-1}}|\widehat{f}(\lambda,\zeta)|^{2}|\mathfrak{c}(\lambda)|^{-2}d\lambda d\sigma(\varsigma).
\end{equation}

Since $e_{\lambda,\zeta}(x)$ is an eigenfunction of $\Delta_{\mathbb{H}}$ with eigenvalue $-\frac{(n-1)^{2}+\lambda^{2}}{4}$, it is easy to check that, for
$f\in C^{\infty}_{0}(\mathbb{B}^{n})$,
\[
\widehat{\Delta_{\mathbb{H}}f}(\lambda,\zeta)=-\frac{(n-1)^{2}+\lambda^{2}}{4}\widehat{f}(\lambda,\zeta).
\]
Therefore, in analogy with the Euclidean setting, we define the fractional
Laplacian on hyperbolic space as follows:
\begin{equation}\label{2.6}
\widehat{(-\Delta_{\mathbb{H}})^{\gamma}f}(\lambda,\zeta)=\left(\frac{(n-1)^{2}+\lambda^{2}}{4}\right)^{\gamma}\widehat{f}(\lambda,\zeta),\;\;\gamma\in \mathbb{R}.
\end{equation}
For more information about fractional
Laplacian on hyperbolic space, we refer to \cite{an,ba}.

\section{Sharp Estimates of Green's function}
In what follows, $a\lesssim b$ will stand for $a\leq C b$  and $a\sim b$ will stand for $C^{-1}b\leq a\leq C b$ with a positive constant $C$.

Let $n\geq2$. Denote by $e^{t\Delta_{\mathbb{H}}}$ the heat kernel on $\mathbb{B}^{n}$. It is well known that $e^{t\Delta_{\mathbb{H}}}$ depends only on $t$ and $\rho(x,y)$. In fact, $e^{t\Delta_{\mathbb{H}}}$  is given explicitly by the following formulas (see e.g. \cite{d,gn}):

\begin{itemize}
  \item If $n=2m$, then
  \begin{equation}\label{3.1}
\begin{split}
e^{t\Delta_{\mathbb{H}}}=&(2\pi)^{-\frac{n+1}{2}}t^{-\frac{1}{2}}e^{-\frac{(n-1)^{2}}{4}t}
\int^{+\infty}_{\rho}\frac{\sinh r}{\sqrt{\cosh r-\cosh\rho}}\left(-\frac{1}{\sinh r}\frac{\partial}{\partial r}\right)^{m}e^{-\frac{r^{2}}{4t}}dr\\
=&\frac{1}{2(2\pi)^{\frac{n+1}{2}}}t^{-\frac{3}{2}}\int^{+\infty}_{\rho}\frac{\sinh r}{\sqrt{\cosh r-\cosh\rho}}\left(-\frac{1}{\sinh r}\frac{\partial}{\partial r}\right)^{m-1}\left(\frac{r}{\sinh r} e^{-\frac{r^{2}}{4t}}\right)dr;
\end{split}
\end{equation}
  \item If $n=2m+1$, then
   \begin{equation}\label{3.2}
  \begin{split}
e^{t\Delta_{\mathbb{H}}}=&2^{-m-1}\pi^{-m-1/2}t^{-\frac{1}{2}}e^{-\frac{(n-1)^{2}}{4}t}\left(-\frac{1}{\sinh \rho}\frac{\partial}{\partial \rho}\right)^{m}e^{-\frac{\rho^{2}}{4t}}\\
=&2^{-m-2}\pi^{-m-1/2}t^{-\frac{3}{2}}e^{-\frac{(n-1)^{2}}{4}t}\left(-\frac{1}{\sinh \rho}\frac{\partial}{\partial \rho}\right)^{m-1}\left(\frac{\rho}{\sinh \rho} e^{-\frac{\rho^{2}}{4t}}\right).
\end{split}
\end{equation}

\end{itemize}

  An explicit expression of Green's function $(-\Delta_{\mathbb{H}}+\lambda)^{-1}$ with $\lambda>-\frac{(n-1)^{2}}{4}$ is given by (see \cite{li,mat})
\begin{equation}\label{3.3}
\begin{split}
(\lambda-\Delta_{\mathbb{H}})^{-1}
=&(2\pi)^{-\frac{n}{2}}(\sinh\rho)^{-\frac{n-2}{2}}e^{-\frac{(n-2)\pi}{2}i}Q^{\frac{n-2}{2}}_{\theta_{n}(\lambda)}(\cosh\rho), \;\; n\geq3,
\end{split}
\end{equation}
where
\begin{equation}\label{3.4}
\begin{split}
\theta_{n}(\lambda)=\sqrt{\lambda+\frac{(n-1)^{2}}{4}}-\frac{1}{2}
\end{split}
\end{equation}
and $Q^{\frac{n-2}{2}}_{\theta_{n}(\lambda)}(\cosh\rho)$ is the Legendre function of second type and satisfies (\cite{er}, Page 155 )
\begin{equation*}
\begin{split}
e^{-i(\pi\gamma)}Q^{\frac{n-2}{2}}_{\theta_{n}(\lambda)}(\cosh\rho)
=&\frac{\Gamma(\frac{n}{2}+\theta_{n}(\lambda))}{2^{\theta_{n}(\lambda)+1}
\Gamma(\theta_{n}(\lambda)+1)\sinh^{\frac{n-2}{2}}\rho}\int^{\pi}_{0}(\cosh\rho+\cos t)^{\frac{n-4}{2}-\theta_{n}(\lambda)}
(\sin t)^{2\theta_{n}(\lambda)+1}dt.
\end{split}
\end{equation*}
Therefore, for $n\geq3$,
\begin{equation}\label{3.5}
\begin{split}
(\lambda-\Delta_{\mathbb{H}})^{-1}
=&\frac{A_{n}}{(\sinh\rho)^{n-2}}\int^{\pi}_{0}(\cosh\rho+\cos t)^{\frac{n-4}{2}-\theta_{n}(\lambda)}
(\sin t)^{2\theta_{n}(\lambda)+1}dt,
\end{split}
\end{equation}
where
$$A_{n}=(2\pi)^{-\frac{n}{2}}\frac{\Gamma(\frac{n}{2}+\theta_{n}(\lambda))}{2^{\theta_{n}(\lambda)+1}
\Gamma(\theta_{n}(\lambda)+1)}.$$

\begin{lemma} Let $n\geq3$.
There holds, for $\lambda>-\frac{(n-1)^{2}}{4}$ and $\rho>0$,
\begin{equation}\label{3.6}
\begin{split}
(\lambda-\Delta_{\mathbb{H}})^{-1}
\lesssim \left(\frac{1}{\sinh\frac{\rho}{2}}\right)^{n-2}
\left(\frac{1}{\cosh\frac{\rho}{2}}\right)^{1+2\sqrt{\lambda+\frac{(n-1)^{2}}{4}}}.
\end{split}
\end{equation}
\end{lemma}
\begin{proof}
Set,  for $0\leq\alpha<1$,
\[
f(\alpha)=\int^{\pi}_{0}(1+\alpha\cos t)^{\frac{n-4}{2}-\theta_{n}(\lambda)}
(\sin t)^{2\theta_{n}(\lambda)+1}dt.
\]
Then $f$ is a continuous function on $[0,1)$ and
\begin{equation*}
\begin{split}
\lim_{\alpha\rightarrow1-}f(\alpha)=&\int^{\pi}_{0}(1+\cos t)^{\frac{n-4}{2}-\theta_{n}(\lambda)}
(\sin t)^{2\theta_{n}(\lambda)+1}dt\\
=&\int^{\pi}_{0}\left(2\sin^{2} \frac{t}{2}\right)^{\frac{n-4}{2}-\theta_{n}(\lambda)}
\left(2\sin \frac{t}{2}\cos\frac{t}{2}\right)^{2\theta_{n}(\lambda)+1}dt\\
=&2^{\frac{n-2}{2}+\theta_{n}(\lambda)+1}\int^{\pi}_{0}\left(\sin \frac{t}{2}\right)^{n-3}
\left(\cos\frac{t}{2}\right)^{2\theta_{n}(\lambda)+1}dt\\
=&2^{\frac{n-2}{2}+\theta_{n}(\lambda)+1}\int^{\pi}_{0}\left(\sin \frac{t}{2}\right)^{n-3}
\left(\cos\frac{t}{2}\right)^{2\sqrt{\lambda+\frac{(n-1)^{2}}{4}}}dt\\
\leq&2^{\frac{n-2}{2}+\theta_{n}(\lambda)+1}\pi.
\end{split}
\end{equation*}
Therefore,  there exists a constant $C>0$ such that for each $\alpha\in[0,1)$,
$|f(\alpha)|=f(\alpha)\leq C$. Thus, by (\ref{3.5}),
\begin{equation*}
\begin{split}
(\lambda-\Delta_{\mathbb{H}})^{-1}=&A_{n}\frac{(\cosh \rho)^{\frac{n-4}{2}-\theta_{n}(\lambda)}}{(\sinh\rho)^{n-2}}f\left(\frac{1}{\cosh\rho}\right)
\\
\lesssim& \frac{(\cosh \rho)^{\frac{n-4}{2}-\theta_{n}(\lambda)}}{(\sinh\rho)^{n-2}}= \frac{(\cosh \rho)^{\frac{n-4}{2}-\theta_{n}(\lambda)}}{(2\sinh\frac{\rho}{2}\cosh\frac{\rho}{2})^{n-2}}\\
\sim&\left(\frac{1}{\sinh\frac{\rho}{2}}\right)^{n-2}
\left(\frac{1}{\cosh\frac{\rho}{2}}\right)^{2+2\theta_{n}(\lambda)}\\
=&\left(\frac{1}{\sinh\frac{\rho}{2}}\right)^{n-2}
\left(\frac{1}{\cosh\frac{\rho}{2}}\right)^{1+2\sqrt{\lambda+\frac{(n-1)^{2}}{4}}}.
\end{split}
\end{equation*}
To get the last inequality, we use the fact  $\cosh\rho\sim\cosh^{2}\frac{\rho}{2}$, $\rho\geq0$.
\end{proof}

\medskip

Next we shall give the estimates of limiting case of Green's function, namely $\lambda=\frac{(n-1)^{2}}{4}$.
We compute, by (\ref{3.1}) and (\ref{3.2}),
\begin{itemize}
  \item If $n=2m$, then
\begin{equation}\label{3.7}
\begin{split}
&\left(-\Delta_{\mathbb{H}}-\frac{(n-1)^{2}}{4}\right)^{-1}=\int^{\infty}_{0}e^{t(\Delta_{\mathbb{H}}+\frac{(n-1)^{2}}{4})}dt\\
=&\frac{1}{2(2\pi)^{\frac{n+1}{2}}}\int^{+\infty}_{\rho}\frac{\sinh r}{\sqrt{\cosh r-\cosh\rho}}\left(-\frac{1}{\sinh r}\frac{\partial}{\partial r}\right)^{m-1}\left(\frac{r}{\sinh r}\int^{\infty}_{0}t^{-\frac{3}{2}}e^{-\frac{r^{2}}{4t}}dt\right)dr\\
=&\frac{\sqrt{\pi}}{(2\pi)^{\frac{n+1}{2}}}\int^{+\infty}_{\rho}\frac{\sinh r}{\sqrt{\cosh r-\cosh\rho}}\left(-\frac{1}{\sinh r}\frac{\partial}{\partial r}\right)^{m-1}\frac{1}{\sinh r}dr
\end{split}
\end{equation}
To get the last equation, we use the fact
\begin{equation}\label{bb3.7}
\int^{\infty}_{0}t^{-\frac{3}{2}}e^{-\frac{r^{2}}{4t}}dt=\frac{2}{r}\int^{\infty}_{0}t^{-\frac{3}{2}}e^{-\frac{1}{t}}dt=\frac{2}{r}
\int^{\infty}_{0}t^{-\frac{1}{2}}e^{-t}dt=\frac{2}{r}\Gamma(1/2)=\frac{2\sqrt{\pi}}{r}.
\end{equation}

  \item If $n=2m+1$, then
  \begin{equation}\label{3.8}
\begin{split}
&\left(-\Delta_{\mathbb{H}}-\frac{(n-1)^{2}}{4}\right)^{-1}=\int^{\infty}_{0}e^{t(\Delta_{\mathbb{H}}+\frac{(n-1)^{2}}{4})}dt\\
=&2^{-m-2}\pi^{-m-1/2}\left(-\frac{1}{\sinh \rho}\frac{\partial}{\partial \rho}\right)^{m-1}\left(\frac{\rho}{\sinh \rho} \int^{\infty}_{0}t^{-\frac{3}{2}}e^{-\frac{\rho^{2}}{4t}}dt\right)\\
=&2^{-m-1}\pi^{-m}\left(-\frac{1}{\sinh \rho}\frac{\partial}{\partial \rho}\right)^{m-1}\frac{1}{\sinh \rho}.
\end{split}
\end{equation}
To get the last equation, we also use (\ref{bb3.7}).
\end{itemize}

\begin{lemma}
Let $k$ be a nonnegative integer. Then there exist  constants $\{a_{i}\}^{k}_{i=0}$ such that
\begin{equation}\label{3.9}
\left(-\frac{1}{\sinh \rho}\frac{\partial}{\partial \rho}\right)^{2k}\frac{1}{\sinh \rho}=\sum^{k}_{i=0}a_{i}\left(\frac{1}{\sinh \rho}\right)^{2i+2k+1}.
\end{equation}
Moreover, $a_{0}=(2k)!$.
\end{lemma}
\begin{proof}
We shall prove  by induction.
It is easy to see that (\ref{3.9}) is valid for $k=0$. Now suppose that equation (\ref{3.9}) is valid for $k=l$,
i.e.
\begin{equation*}
\left(-\frac{1}{\sinh \rho}\frac{\partial}{\partial \rho}\right)^{2l}\frac{1}{\sinh \rho}=\sum^{l}_{i=0}a_{i}\left(\frac{1}{\sinh \rho}\right)^{2i+2l+1}
\end{equation*}
and $a_{0}=(2l)!$. Then
 \begin{equation*}
\begin{split}
&\left(-\frac{1}{\sinh \rho}\frac{\partial}{\partial \rho}\right)^{2l+2}\frac{1}{\sinh \rho}
\\
=&\left(-\frac{1}{\sinh \rho}\frac{\partial}{\partial \rho}\right)^{2}\sum^{l}_{i=0}a_{i}\left(\frac{1}{\sinh \rho}\right)^{2i+2l+1}\\
=&\sum^{l}_{i=0}(2i+2l+1)a_{i}\left[-\left(\frac{1}{\sinh \rho}\right)^{2i+2l+3}+(2i+2l+3)\left(\frac{1}{\sinh \rho}\right)^{2i+2l+5}\cosh^{2} \rho\right]\\
=&\sum^{l}_{i=0}(2i+2l+1)a_{i}\left[-\left(\frac{1}{\sinh \rho}\right)^{2i+2l+3}+(2i+2l+3)\left(\frac{1}{\sinh \rho}\right)^{2i+2l+5}(1+\sinh^{2} \rho)\right]\\
=&\sum^{l}_{i=0}(2i+2l+1)a_{i}\left[(2i+2l+2)\left(\frac{1}{\sinh \rho}\right)^{2i+2l+3}+(2i+2l+3)\left(\frac{1}{\sinh \rho}\right)^{2i+2l+5}\right]\\
=&\sum^{l+1}_{i=0}a'_{i}\left(\frac{1}{\sinh \rho}\right)^{2i+2l+3},
\end{split}
\end{equation*}
where
 \begin{equation*}
\begin{split}
a'_{0}=&a_{0}(2l+1)(2l+2)=(2l+2)!;\\
a'_{i}=&(2i+2l+1)(2i+2l+2)a_{i}+(2i+2l-1)(2i+2l+1)a_{i-1},\;\; i=1,2,\cdots,l;\\
a'_{l+1}=&(4l+1)(4l+3)a_{l}.
\end{split}
\end{equation*}
The desired result follows.
\end{proof}

\begin{lemma}
Let $m$ be a nonnegative integer. There holds, for $\rho>0$,
\begin{equation}\label{3.10}
\left|\left(-\frac{1}{\sinh \rho}\frac{\partial}{\partial \rho}\right)^{m}\frac{1}{\sinh \rho}\right|\lesssim \left(\frac{1}{\sinh\frac{\rho}{2}}\right)^{2m+1}
\frac{1}{\cosh\frac{\rho}{2}}.
\end{equation}
\end{lemma}
\begin{proof} If $m$ is even, namely $m=2k$ for
some nonnegative integer $k$, then by Lemma 3.2,
\begin{equation*}
\begin{split}
\left|\left(-\frac{1}{\sinh \rho}\frac{\partial}{\partial \rho}\right)^{m}\frac{1}{\sinh \rho}\right|=&\left|\left(-\frac{1}{\sinh \rho}\frac{\partial}{\partial \rho}\right)^{2k}\frac{1}{\sinh \rho}\right|\\
\lesssim&
\sum^{k}_{i=0}\left(\frac{1}{\sinh \rho}\right)^{2i+2k+1}\\
\lesssim&\left(\frac{1}{\sinh \rho}\right)^{4k+1}+\left(\frac{1}{\sinh \rho}\right)^{2k+1}\\
\sim&\left(\frac{1}{\sinh\frac{\rho}{2}}\right)^{4k+1}\frac{1}{\cosh\frac{\rho}{2}}\left[
\left(\frac{1}{\cosh\frac{\rho}{2}}\right)^{4k}+
\frac{\sinh^{2k}\frac{\rho}{2}}{\cosh^{2k}\frac{\rho}{2}}\right]\\
\lesssim&
\left(\frac{1}{\sinh\frac{\rho}{2}}\right)^{4k+1}\frac{1}{\cosh\frac{\rho}{2}}.
\end{split}
\end{equation*}

If $m$ is odd, namely $m=2k+1$ for
some nonnegative integer $k$. Also by Lemma 3.2,
\begin{equation}\label{3.12}
\begin{split}
&\left|\left(-\frac{1}{\sinh \rho}\frac{\partial}{\partial \rho}\right)^{m}\frac{1}{\sinh \rho}\right|=\left|\left(-\frac{1}{\sinh \rho}\frac{\partial}{\partial \rho}\right)^{2k+1}\frac{1}{\sinh \rho}\right|\\
=&\left|\sum^{k}_{i=0}a_{i}(2i+2k+1)\left(\frac{1}{\sinh \rho}\right)^{2i+2k+3}\cosh\rho\right|\\
\lesssim&
\sum^{k}_{i=0}\left(\frac{1}{\sinh \rho}\right)^{2i+2k+3}\cosh\rho\\
\lesssim&\left(\frac{1}{\sinh \rho}\right)^{4k+3}\cosh\rho+\left(\frac{1}{\sinh \rho}\right)^{2k+3}\cosh\rho\\
\sim&\left(\frac{1}{\sinh\frac{\rho}{2}}\right)^{4k+3}\frac{1}{\cosh\frac{\rho}{2}}\left[
\frac{\cosh\rho}{\cosh^{4k+2}\frac{\rho}{2}}+
\frac{\sinh^{2k}\frac{\rho}{2}}{\cosh^{2k+2}\frac{\rho}{2}}\cosh\rho\right].
\end{split}
\end{equation}
Notice that $\cosh\rho\sim\cosh^{2}\frac{\rho}{2}$, $\rho\geq0$.
We have, by (\ref{3.12}),
\begin{equation*}
\begin{split}
\left|\left(-\frac{1}{\sinh \rho}\frac{\partial}{\partial \rho}\right)^{m}\frac{1}{\sinh \rho}\right|
\lesssim&\left(\frac{1}{\sinh\frac{\rho}{2}}\right)^{4k+3}\frac{1}{\cosh\frac{\rho}{2}}.
\end{split}
\end{equation*}
These complete the proof.
\end{proof}

\begin{lemma} Let $n\geq3$.
There holds, for $\rho>0$,
\begin{equation}\label{3.13}
\left(-\Delta_{\mathbb{H}}-\frac{(n-1)^{2}}{4}\right)^{-1}\lesssim \left(\frac{1}{\sinh\frac{\rho}{2}}\right)^{n-2}
\frac{1}{\cosh\frac{\rho}{2}}.
\end{equation}
\end{lemma}
\begin{proof}
If  $n$ is even, namely $n=2m$ for
some positive integer $m\geq2$.  Then by (\ref{3.7}) and Lemma 3.3,
\begin{equation*}
\begin{split}
\left(-\Delta_{\mathbb{H}}-\frac{(n-1)^{2}}{4}\right)^{-1}\lesssim&
\int^{+\infty}_{\rho}\frac{\sinh r}{\sqrt{\cosh r-\cosh\rho}}\left(\frac{1}{\sinh\frac{r}{2}}\right)^{2m-1}
\frac{1}{\cosh\frac{r}{2}}dr\\
\leq&\frac{1}{\cosh\frac{\rho}{2}}\int^{+\infty}_{\rho}\frac{\sinh r}{\sqrt{\cosh r-\cosh\rho}}\left(\frac{1}{\sinh\frac{r}{2}}\right)^{2m-1}dr.
\end{split}
\end{equation*}
Using the substitution $t=\sqrt{\cosh r-\cosh\rho}/\sqrt{2}=\sqrt{\sinh^{2} \frac{r}{2}-\sinh^{2}\frac{\rho}{2}}$, we have
\begin{equation*}
\begin{split}
\left(-\Delta_{\mathbb{H}}-\frac{(n-1)^{2}}{4}\right)^{-1}\lesssim&\frac{1}{\cosh\frac{\rho}{2}}\int^{+\infty}_{0}
\left(\frac{1}{t^{2}+\sinh^{2}\frac{\rho}{2}}\right)^{\frac{2m-1}{2}}dt\\
=&\frac{1}{\cosh\frac{\rho}{2}}\left(\frac{1}{\sinh\frac{\rho}{2}}\right)^{2m-2}\int^{+\infty}_{0}
\left(\frac{1}{t^{2}+1}\right)^{\frac{2m-1}{2}}dt\\
\sim&\frac{1}{\cosh\frac{\rho}{2}}\left(\frac{1}{\sinh\frac{\rho}{2}}\right)^{n-2}.
\end{split}
\end{equation*}

If  $n$ is odd, namely $n=2m+1$ for
some positive integer $m\geq1$.
Then by (\ref{3.8}) and  Lemma 3.3,
\begin{equation*}
\begin{split}
\left(-\Delta_{\mathbb{H}}-\frac{(n-1)^{2}}{4}\right)^{-1}\sim&\left(-\frac{1}{\sinh \rho}\frac{\partial}{\partial \rho}\right)^{m-1}\frac{1}{\sinh \rho}
\lesssim\left(\frac{1}{\sinh\frac{\rho}{2}}\right)^{n-2}\frac{1}{\cosh\frac{\rho}{2}}.
\end{split}
\end{equation*}
The desired result follows.
\end{proof}

\section{Proofs of Theorem 1.1 and Theorem 1.2}
Firstly, we recall the sharp Hardy-Littlewood-Sobolev inequality on $\mathbb{R}^{n}$ (see \cite{lie}).
\begin{theorem}
Let $0<\lambda<n$ and $p=\frac{2n}{2n-\lambda}$. Then for $f,g\in L^{p}(\mathbb{R}^{n})$,
\[
\left|\int_{\mathbb{R}^{n}}\int_{\mathbb{R}^{n}}\frac{f(x)g(y)}{|x-y|^{\lambda}}dxdy\right|\leq C_{n,\lambda}\|f\|_{p}\|g\|_{p}
\]
with equality if and only if
$g=cf=c'(\gamma^{2}+|x-a|^{2})^{-(2n-\lambda)2}, a\in\mathbb{R}^{n}, \gamma\neq0,$ where $C_{n,\lambda}$ is given by (\ref{1.10})
\end{theorem}

Now we can prove a sharp Hardy-Littlewood-Sobolev inequality on $\mathbb{B}^{n}$.\\

\textbf{Proof of Theorem 1.2}

We have, by (\ref{2.2}),
\begin{equation*}
\begin{split}
&\int_{\mathbb{B}^{n}}\int_{\mathbb{B}^{n}}\frac{f(x)g(y)}{\left(2\sinh\frac{\rho(T_{y}(x))}{2}\right)^{\lambda}}dV_{x}dV_{y}\\
=&\int_{\mathbb{B}^{n}}\int_{\mathbb{B}^{n}}f(x)\left(\frac{2}{1-|x|^{2}}\right)^{n}
\left(\frac{2|x-a|}{\sqrt{(1-|a|^{2})(1-|x|^{2})}}\right)^{-\lambda}g(y)\left(\frac{2}{1-|y|^{2}}\right)^{n}dxdy\\
=&\int_{\mathbb{B}^{n}}\int_{\mathbb{B}^{n}}f(x)\left(\frac{2}{1-|x|^{2}}\right)^{n-\frac{\lambda}{2}}
|x-y|^{-\lambda}g(y)\left(\frac{2}{1-|y|^{2}}\right)^{n-\frac{\lambda}{2}}dxdy.
\end{split}
\end{equation*}
Set $\widetilde{f}=f(x)\left(\frac{2}{1-|x|^{2}}\right)^{n-\frac{\lambda}{2}}$ and $\widetilde{g}=g(y)\left(\frac{2}{1-|y|^{2}}\right)^{n-\frac{\lambda}{2}}$.
Then by Theorem 4.1,
\begin{equation*}
\begin{split}
\left|\int_{\mathbb{B}^{n}}\int_{\mathbb{B}^{n}}\frac{f(x)g(y)}{\left(2\sinh\frac{\rho(T_{y}(x))}{2}\right)^{\lambda}}dV_{x}dV_{y}\right|=&
\left|\int_{\mathbb{B}^{n}}\int_{\mathbb{B}^{n}}\widetilde{f}(x)
|x-y|^{-\lambda}\widetilde{g}(y)dxdy\right|\\
\leq&\int_{\mathbb{B}^{n}}\int_{\mathbb{B}^{n}}|\widetilde{f}(x)|\cdot
|x-y|^{-\lambda}\cdot|\widetilde{g}(y)|dxdy\\
\leq&C_{n,\lambda}\left(\int_{\mathbb{B}^{n}}|\widetilde{f}(x)|^{\frac{2n}{2n-\lambda}}dx\right)^{\frac{2n-\lambda}{2n}}\cdot
\left(\int_{\mathbb{B}^{n}}|\widetilde{g}(y)|^{\frac{2n}{2n-\lambda}}dy\right)^{\frac{2n-\lambda}{2n}}\\
=&C_{n,\lambda}\left(\int_{\mathbb{B}^{n}}|f|^{\frac{2n}{2n-\lambda}}dV\right)^{\frac{2n-\lambda}{2n}}\cdot
\left(\int_{\mathbb{B}^{n}}|g(y)|^{\frac{2n}{2n-\lambda}}dV\right)^{\frac{2n-\lambda}{2n}}.
\end{split}
\end{equation*}
Also by Theorem 4.1, it is easy to see that $C_{n,\lambda}$ is sharp and there is no nonezero extreme function.
The proof of Theorem 1.2 is thereby completed.
\\

Before the proof of Theorem 1.1, we need the following Lemma.
\begin{lemma}
Let $0<\alpha<n$, $0<\beta<n$ and $0<\alpha+\beta<n$. Then
\begin{equation*}
\begin{split}
&\int_{\mathbb{B}^{n}}\left(\sinh\frac{\rho(x,y)}{2}\right)^{\alpha-n}\left(\cosh\frac{\rho(x,y)}{2}\right)^{-\alpha-\beta}
\left(\sinh\frac{\rho(x,z)}{2}\right)^{\beta-n}dV_{x}\\
=&\left\{\left[\left(\sinh\frac{\rho}{2}\right)^{\alpha-n}\left(\cosh\frac{\rho}{2}\right)^{-\alpha-\beta}\right]\ast
\left(\sinh\frac{\rho}{2}\right)^{\beta-n}\right\}(T_{z}(y))\\
\leq& 2^{n}
\frac{\gamma(\alpha)\gamma(\beta)}{\gamma(\alpha+\beta)}
\left(\sinh\frac{\rho(y,z)}{2}\right)^{\alpha+\beta-n}\left(\cosh\frac{\rho(y,z)}{2}\right)^{-\alpha},
\end{split}
\end{equation*}
where
$
\gamma(\alpha)=\pi^{n/2}2^{\alpha}\Gamma(\alpha/2)/\Gamma(\frac{n}{2}-\frac{\alpha}{2}).
$
\end{lemma}
\begin{proof}
We firstly show
\begin{equation}\label{b4.1}
\begin{split}
&\int_{\mathbb{B}^{n}}\left(\sinh\frac{\rho(x,y)}{2}\right)^{\alpha-n}\left(\cosh\frac{\rho(x,y)}{2}\right)^{-\alpha-\beta}
\left(\sinh\frac{\rho(x,z)}{2}\right)^{\beta-n}dV_{x}\\
=&\left\{\left[\left(\sinh\frac{\rho}{2}\right)^{\alpha-n}\left(\cosh\frac{\rho}{2}\right)^{-\alpha-\beta}\right]\ast
\left(\sinh\frac{\rho}{2}\right)^{\beta-n}\right\}(T_{z}(y)).
\end{split}
\end{equation}
In fact, using the following  identity (see \cite{liup}, (3.13)),
\begin{equation*}
\begin{split}
|T_{z}(T_{y}(x))|=|T_{x}(T_{y}(z))|, \;\;x,y,z\in\mathbb{B}^{n},
\end{split}
\end{equation*}
we have, by the M\"obius shift invariance,
\begin{equation*}
\begin{split}
&\int_{\mathbb{B}^{n}}\left(\sinh\frac{\rho(x,y)}{2}\right)^{\alpha-n}\left(\cosh\frac{\rho(x,y)}{2}\right)^{-\alpha-\beta}
\left(\sinh\frac{\rho(x,z)}{2}\right)^{\beta-n}dV_{x}\\
=&\int_{\mathbb{B}^{n}}\left(\sinh\frac{\rho(T_{y}(x),y)}{2}\right)^{\alpha-n}\left(\cosh\frac{\rho(T_{y}(x),y)}{2}\right)^{-\alpha-\beta}
\left(\sinh\frac{\rho(T_{y}(x),z)}{2}\right)^{\beta-n}dV_{x}\\
=&\int_{\mathbb{B}^{n}}\left(\sinh\frac{\rho(T_{y}(T_{y}(x)))}{2}\right)^{\alpha-n}\left(\cosh\frac{\rho(T_{y}(T_{y}(x)))}{2}\right)^{-\alpha-\beta}
\left(\sinh\frac{\rho(T_{z}(T_{y}(x)))}{2}\right)^{\beta-n}dV_{x}\\
=&\int_{\mathbb{B}^{n}}\left(\sinh\frac{\rho(x)}{2}\right)^{\alpha-n}\left(\cosh\frac{\rho(x)}{2}\right)^{-\alpha-\beta}
\left(\sinh\frac{\rho(T_{x}(T_{z}(y)))}{2}\right)^{\beta-n}dV_{x}\\
=&\int_{\mathbb{B}^{n}}\left(\sinh\frac{\rho(x)}{2}\right)^{\alpha-n}\left(\cosh\frac{\rho(x)}{2}\right)^{-\alpha-\beta}
\left(\sinh\frac{\rho(T_{T_{z}(y)}(x)}{2}\right)^{\beta-n}dV_{x}\\
=&\left\{\left[\left(\sinh\frac{\rho}{2}\right)^{\alpha-n}\left(\cosh\frac{\rho}{2}\right)^{-\alpha-\beta}\right]\ast
\left(\sinh\frac{\rho}{2}\right)^{\beta-n}\right\}(T_{z}(y)).
\end{split}
\end{equation*}

To finish the proof, it is enough to show
\begin{equation}\label{b4.2}
\begin{split}
&\left\{\left[\left(\sinh\frac{\rho}{2}\right)^{\alpha-n}\left(\cosh\frac{\rho}{2}\right)^{-\alpha-\beta}\right]\ast
\left(\sinh\frac{\rho}{2}\right)^{\beta-n}\right\}(y)\\
\leq& 2^{n}
\frac{\gamma(\alpha)\gamma(\beta)}{\gamma(\alpha+\beta)}
\left(\sinh\frac{\rho(y)}{2}\right)^{\alpha+\beta-n}\left(\cosh\frac{\rho(y)}{2}\right)^{-\alpha}.
\end{split}
\end{equation}
Notice that, for $0<\alpha<n$, $0<\beta<n$ and $0<\alpha+\beta<n$, we have  (see e.g. \cite{s})
\begin{equation}\label{4.1}
\begin{split}
\int_{\mathbb{R}^{n}}|x|^{\alpha-n}|y-x|^{\beta-n}dx=
\frac{\gamma(\alpha)\gamma(\beta)}{\gamma(\alpha+\beta)}|y|^{\alpha+\beta-n}.
\end{split}
\end{equation}
Therefore, by (\ref{2.2}) and (\ref{4.1}),
\begin{equation*}
\begin{split}
&\left\{\left[\left(\sinh\frac{\rho}{2}\right)^{\alpha-n}\left(\cosh\frac{\rho}{2}\right)^{-\alpha-\beta}\right]\ast
\left(\sinh\frac{\rho}{2}\right)^{\beta-n}\right\}(y)\\
=&\int_{\mathbb{B}^{n}}\left(\sinh\frac{\rho(x)}{2}\right)^{\alpha-n}\left(\cosh\frac{\rho(x)}{2}\right)^{-\alpha-\beta}
\left(\sinh\frac{\rho(T_{y}(x))}{2}\right)^{\beta-n}
\left(\frac{2}{1-|x|^{2}}\right)^{n}dx\\
=&\int_{\mathbb{B}^{n}}\left(\frac{|x|}{\sqrt{1-|x|^{2}}}\right)^{\alpha-n}\left(\frac{1}{\sqrt{1-|x|^{2}}}\right)^{-\alpha-\beta}
\left(\frac{|x-y|}{\sqrt{(1-|y|^{2})(1-|x|^{2})}}\right)^{\beta-n}
\left(\frac{2}{1-|x|^{2}}\right)^{n}dx\\
=&\frac{2^{n}}{(1-|y|^{2})^{(\beta-n)/2}}\int_{\mathbb{B}^{n}}|x|^{\alpha-n}|x-y|^{\beta-n}dx\\
\leq&\frac{2^{n}}{(1-|y|^{2})^{(\beta-n)/2}}\int_{\mathbb{R}^{n}}|x|^{\alpha-n}|x-y|^{\beta-n}dx\\
=&\frac{2^{n}}{(1-|y|^{2})^{(\beta-n)/2}}\cdot\frac{\gamma(\alpha)\gamma(\beta)}{\gamma(\alpha+\beta)}
|y|^{\alpha+\beta-n}\\
=&2^{n}
\frac{\gamma(\alpha)\gamma(\beta)}{\gamma(\alpha+\beta)}
\left(\sinh\frac{\rho(y)}{2}\right)^{\alpha+\beta-n}\left(\cosh\frac{\rho(y)}{2}\right)^{-\alpha}.
\end{split}
\end{equation*}
The desired result follows.
\end{proof}

\begin{lemma}
Let $k$ be a positive integer with  $2\leq k<\frac{n}{2}$. The kernel $Q^{-1}_{k}(\rho)$ satisfies
 \begin{equation}\label{4.2}
\begin{split}
Q^{-1}_{k}(\rho)\lesssim \left(\frac{1}{\sinh\frac{\rho}{2}}\right)^{n-2k},\;\; \rho>0.
\end{split}
 \end{equation}
\end{lemma}
\begin{proof}
Recall that
 \begin{equation*}
\begin{split}
Q_{k}=\left(-\Delta_{\mathbb{H}}-\frac{(n-1)^{2}}{4}\right)(P_{1}+2)\cdot\cdots\cdot(P_{1}+k(k-1)).
\end{split}
 \end{equation*}
We have, by (\ref{2.4}),
 \begin{equation*}
\begin{split}
Q^{-1}_{k}=\left(-\Delta_{\mathbb{H}}-\frac{(n-1)^{2}}{4}\right)^{-1}\ast(P_{1}+2)^{-1}\ast\cdots\ast(P_{1}+k(k-1))^{-1},
\end{split}
 \end{equation*}
where, by Lemma 3.1,  $(P_{1}+i(i-1))^{-1} (2\leq i\leq k)$ satisfies
\begin{equation}\label{4.3}
\begin{split}
(P_{1}+i(i-1))^{-1}=&\left(i(i-1)-\frac{n(n-2)}{4}-\Delta_{\mathbb{H}}\right)^{-1}\\
=&\left((i-1/2)^{2}-\frac{(n-1)^{2}}{4}-\Delta_{\mathbb{H}}\right)^{-1}\\
&\lesssim \left(\frac{1}{\sinh\frac{\rho}{2}}\right)^{n-2}
\left(\frac{1}{\cosh\frac{\rho}{2}}\right)^{2i}.
\end{split}
\end{equation}

We shall prove (\ref{4.2}) by induction. For $k=2$, we have, by  (\ref{4.3}) and Lemma 4.2,
 \begin{equation*}
\begin{split}
Q^{-1}_{2}=&\left(-\Delta_{\mathbb{H}}-\frac{(n-1)^{2}}{4}\right)^{-1}\ast(P_{1}+2)^{-1}\\
\lesssim&\left[\left(\frac{1}{\sinh\frac{\rho}{2}}\right)^{n-2}
\frac{1}{\cosh\frac{\rho}{2}}\right]\ast\left[\left(\frac{1}{\sinh\frac{\rho}{2}}\right)^{n-2}
\left(\frac{1}{\cosh\frac{\rho}{2}}\right)^{4}\right]\\
\leq&\left(\frac{1}{\sinh\frac{\rho}{2}}\right)^{n-2}
\ast\left[\left(\frac{1}{\sinh\frac{\rho}{2}}\right)^{n-2}
\left(\frac{1}{\cosh\frac{\rho}{2}}\right)^{4}\right]\\
\lesssim&\left(\frac{1}{\sinh\frac{\rho}{2}}\right)^{n-4}\left(\frac{1}{\cosh\frac{\rho}{2}}\right)^{2}\\
\leq&\left(\frac{1}{\sinh\frac{\rho}{2}}\right)^{n-4}.
\end{split}
 \end{equation*}

Now suppose that equation (\ref{4.3}) is valid for $k=l$,
i.e. $Q^{-1}_{l}(\rho)\lesssim \left(\frac{1}{\sinh\frac{\rho}{2}}\right)^{n-2l}.$ Then by (\ref{4.3}) and Lemma 4.2,
\begin{equation*}
\begin{split}
Q^{-1}_{l+1}(\rho)=&Q^{-1}_{l}(\rho)\ast (P_{1}+(l+1)l)\\
\lesssim& \left(\frac{1}{\sinh\frac{\rho}{2}}\right)^{n-2l}\ast\left[\left(\frac{1}{\sinh\frac{\rho}{2}}\right)^{n-2}
\left(\frac{1}{\cosh\frac{\rho}{2}}\right)^{2l+2}\right]\\
\lesssim&\left(\frac{1}{\sinh\frac{\rho}{2}}\right)^{n-2l-2}
\left(\frac{1}{\cosh\frac{\rho}{2}}\right)^{2}\\
\leq&\left(\frac{1}{\sinh\frac{\rho}{2}}\right)^{n-2l-2}.
\end{split}
 \end{equation*}
The desired result follows.
\end{proof}

\textbf{Proof of Theorem 1.1}
We first prove, for some positive constant $C>0$,
\begin{equation}\label{4.4}
\int_{\mathbb{B}^{n}}(Q_{k}u)udV\geq C\left(\int_{\mathbb{B}^{n}}|u|^{\frac{2n}{n-2k}}dV\right)^{\frac{n-2k}{n}},\;\; u\in C^{\infty}_{0}(\mathbb{B}^{n}).
\end{equation}
Without loss of generality, we may assume $u\geq0$.

\medskip

By Lemma 4.3 and Theorem 1.2, we have
\begin{equation}\label{4.5}
\begin{split}
\int_{\mathbb{B}^{n}}|(Q^{-\frac{1}{2}}_{k}f)(x)|^{2}dV=&\int_{\mathbb{B}^{n}}f(x)(Q^{-1}_{k}f)(x)dV\\
\lesssim&
\int_{\mathbb{B}^{n}}\int_{\mathbb{B}^{n}}\frac{f(x)f(y)}{\left(2\sinh\frac{\rho(T_{y}(x))}{2}\right)^{n-2k}}dV_{x}dV_{y}\\
&\lesssim
\left(\int_{\mathbb{B}^{n}}|f(x)|^{\frac{2n}{n+2k}}dV\right)^{\frac{n+2k}{n}}.
\end{split}
 \end{equation}

On the other hand,
\begin{equation}\label{4.6}
\begin{split}
\left|\int_{\mathbb{B}^{n}}f(x)g(x)dV\right|^{2}=&\left|\int_{\mathbb{B}^{n}}(Q^{\frac{1}{2}}_{k}f)(x)(Q^{-\frac{1}{2}}_{k}g)(x)dV\right|^{2}\\
\leq&\int_{\mathbb{B}^{n}}|(Q^{\frac{1}{2}}_{k}f)(x)|^{2}dV\cdot\int_{\mathbb{B}^{n}}|(Q^{-\frac{1}{2}}_{k}g)(x)|^{2}dV.
\end{split}
 \end{equation}
Combing (\ref{4.5}) and (\ref{4.6}) yields
\begin{equation}\label{4.7}
\begin{split}
\left|\int_{\mathbb{B}^{n}}f(x)g(x)dV\right|^{2}\lesssim&\int_{\mathbb{B}^{n}}|(Q^{\frac{1}{2}}_{k}f)(x)|^{2}dV
\left(\int_{\mathbb{B}^{n}}|g(x)|^{\frac{2n}{n+2k}}dV\right)^{\frac{n+2k}{n}}\\
=&\int_{\mathbb{B}^{n}}Q_{k}f(x)\cdot f(x)dV
\left(\int_{\mathbb{B}^{n}}|g(x)|^{\frac{2n}{n+2k}}dV\right)^{\frac{n+2k}{n}}.
\end{split}
 \end{equation}
Taking $g=f^{\frac{n+2k}{n-2k}}$, we have, by (\ref{4.7}),
\begin{equation}
\begin{split}
\left(\int_{\mathbb{B}^{n}}|f(x)|^{\frac{2n}{n-2k}}dV\right)^{2}\lesssim&\int_{\mathbb{B}^{n}}Q_{k}f(x)\cdot f(x)dV
\left(\int_{\mathbb{B}^{n}}|f(x)|^{\frac{2n}{n-2k}}dV\right)^{\frac{n+2k}{n}}.\\
\end{split}
 \end{equation}
Therefore,
\begin{equation}
\begin{split}
\left(\int_{\mathbb{B}^{n}}|f(x)|^{\frac{2n}{n-2k}}dV\right)^{\frac{n-2k}{n}}\lesssim\int_{\mathbb{B}^{n}}Q_{k}f(x)\cdot f(x)dV.
\end{split}
 \end{equation}

Now we shall prove inequality (\ref{1.9}).
Notice that, by Plancherel formula, (\ref{2.6}) and (\ref{1.3}),
\begin{equation}\label{4.10},
\begin{split}
\int_{\mathbb{B}^{n}}Q_{k}f(x)\cdot f(x)dV=&D_{n}\int^{+\infty}_{-\infty}\int_{\mathbb{S}^{n-1}}\frac{\lambda^{2}}{4}\prod^{k}_{i=2}\frac{(2i-1)^{2}+\lambda^{2}}{4}|\widehat{f}(\lambda,\zeta)|^{2}|\mathfrak{c}(\lambda)|^{-2}d\lambda d\sigma(\varsigma)\\
\geq&\prod^{k}_{i=2}\frac{(2i-1)^{2}}{4}D_{n}\int^{+\infty}_{-\infty}\int_{\mathbb{S}^{n-1}}\frac{\lambda^{2}}{4}
|\widehat{f}(\lambda,\zeta)|^{2}|\mathfrak{c}(\lambda)|^{-2}d\lambda d\sigma(\varsigma)\\
=&\prod^{k}_{i=2}\frac{(2i-1)^{2}}{4}\left(\int_{\mathbb{B}^{n}}|\nabla_{\mathbb{H}}u|^{2}dV-\frac{(n-1)^{2}}{4}\int_{\mathbb{B}^{n}}u^{2}dV\right)\\
\geq&C\left(\int_{\mathbb{B}^{n}}|u|^{p}dV\right)^{\frac{2}{p}},\;\; 2<p\leq\frac{2n}{n-2}.
\end{split}
 \end{equation}
Therefore, for $2<p\leq\frac{2n}{n-2k}$, choose $\widetilde{p}\in(2,\frac{2n}{n-2}]$ such that $\widetilde{p}<p$. By H\"older inequality and (\ref{4.10}),
\begin{equation}
\begin{split}
\int_{\mathbb{B}^{n}}|u|^{p}dV=&\int_{\mathbb{B}^{n}}|u|^{s}|u|^{t}dV\\
\leq&\left(\int_{\mathbb{B}^{n}}|u|^{\widetilde{p}}dV\right)^{\frac{s}{\widetilde{p}}}
\left(\int_{\mathbb{B}^{n}}|u|^{\frac{2n}{n-2k}}dV\right)^{\frac{(n-2k)t}{2n}}\\
\leq&\left(\int_{\mathbb{B}^{n}}Q_{k}f(x)\cdot f(x)dV\right)^{\frac{s}{2}}\left(\int_{\mathbb{B}^{n}}Q_{k}f(x)\cdot f(x)dV\right)^{\frac{t}{2}}\\
=&\left(\int_{\mathbb{B}^{n}}Q_{k}f(x)\cdot f(x)dV\right)^{\frac{p}{2}},
\end{split}
 \end{equation}
where $s=(1-\frac{n-2k}{2n}p)(\frac{1}{\widetilde{p}}-\frac{n-2k}{2n})^{-1}$ and $t=p-s$.
The desired result follows.

\section{proof of Theorem 1.4}
It has been shown by Liu (see \cite{liu}, Theorem 2.3):
\begin{equation}\label{b7.1}
\left(\frac{1-|x|^{2}}{2}\right)^{k+\frac{n}{2}}(-\Delta)^{k}\left[\left(\frac{1-|x|^{2}}{2}\right)^{k-\frac{n}{2}}f\right]=P_{k}f,\;\; f\in C^{\infty}_{0}(\mathbb{B}^{n}),\;\;k\in\mathbb{N}.
\end{equation}
Similarly, in term of the half space model $\mathbb{H}^{n}$
we also have the following:
\begin{lemma}
Let $k$ be a positive integer. There holds,  for each $f\in C^{\infty}_{0}(\mathbb{H}^{n})$,
\begin{equation}\label{5.1}
x^{\frac{n}{2}+k}_{1}(-\Delta)^{k}(x^{k-\frac{n}{2}}_{1}f)=P_{k}f.
\end{equation}
\end{lemma}
\begin{proof}
 It is enough to show
\begin{equation}\label{5.2}
x^{\frac{n}{2}+k}_{1}\Delta^{k}(x^{k-\frac{n}{2}}_{1}f)=\prod^{k}_{i=1}\left[\Delta_{\mathbb{H}}+\frac{(n-2i)(n+2i-2)}{4}\right]f.
\end{equation}
We shall prove (\ref{5.2}) by induction.

\medskip

A simple calculation shows, for each $\alpha\in\mathbb{R}$ and $f\in C^{\infty}_{0}(\mathbb{H}^{n})$, there holds
\begin{equation}\label{5.3}
\begin{split}
x^{\alpha+2}_{1}\Delta(x^{-\alpha}_{1}f)=&x^{2}_{1}\Delta f-2\alpha x_{1}\frac{\partial f}{\partial x_{1}}+\alpha(\alpha+1)f\\
=&=[\Delta_{\mathbb{H}}+\alpha(\alpha+1)]f+(n-2-2\alpha)x_{1}\frac{\partial f}{\partial x_{1}}.
\end{split}
\end{equation}
If we choose $\alpha=\frac{n-2}{2}$ in (\ref{5.3}), then
\begin{equation}\label{5.4}
x^{\frac{n}{2}+1}_{1}\Delta(x^{1-\frac{n}{2}}_{1}f)=\left[\Delta_{\mathbb{H}}+\frac{n(n-2)}{4}\right]f.
\end{equation}

 Now suppose that equation (\ref{5.2}) is valid for $k=l$,
i.e.
\begin{equation}\label{5.5}
x^{\frac{n}{2}+l}_{1}\Delta^{l}(x^{l-\frac{n}{2}}_{1}f)=\prod^{l}_{i=1}\left[\Delta_{\mathbb{H}}+\frac{(n-2i)(n+2i-2)}{4}\right]f.
\end{equation}
We note it is easy to check
\begin{equation}\label{5.6}
\Delta^{l+1}(x_{1}g)=x_{1}\Delta^{l+1}g+2(l+1)\frac{\partial }{\partial x_{1}}\Delta^{l}g,\;\;g\in C^{\infty}_{0}(\mathbb{H}^{n}).
\end{equation}
Therefore, by  (\ref{5.5}) and (\ref{5.6}),
\begin{equation}\label{5.8}
\begin{split}
&x^{\frac{n}{2}+l+1}_{1}\Delta^{l+1}(x^{l+1-\frac{n}{2}}_{1}f)=x^{\frac{n}{2}+l+1}_{1}\Delta^{l+1}(x_{1}\cdot x^{l-\frac{n}{2}}_{1}f)\\
=&x^{\frac{n}{2}+l+1}_{1}x_{1}\Delta^{l+1}( x^{l-\frac{n}{2}}_{1}f)+2(l+1)x^{\frac{n}{2}+l+1}_{1}\frac{\partial }{\partial x_{1}}\Delta^{l}( x^{l-\frac{n}{2}}_{1}f)\\
=&x^{\frac{n}{2}+l+2}_{1}\Delta\left\{x^{-\frac{n}{2}-l}_{1}\prod^{l}_{i=1}\left[\Delta_{\mathbb{H}}+\frac{(n-2i)(n+2i-2)}{4}\right]f\right\}+\\
&2(l+1)x^{\frac{n}{2}+l+1}_{1}\frac{\partial }{\partial x_{1}}\left\{x^{-\frac{n}{2}-l}_{1}\prod^{l}_{i=1}\left[\Delta_{\mathbb{H}}+\frac{(n-2i)(n+2i-2)}{4}\right]f\right\}.
\end{split}
\end{equation}
By (\ref{5.3}), we have
\begin{equation}\label{5.9}
\begin{split}
&x^{\frac{n}{2}+l+2}_{1}\Delta\left\{x^{-\frac{n}{2}-l}_{1}\prod^{l}_{i=1}\left[\Delta_{\mathbb{H}}+\frac{(n-2i)(n+2i-2)}{4}\right]f\right\}
\\
=&\left[\Delta_{\mathbb{H}}+\frac{(n+2l)(n+2l+2)}{4}\right]\prod^{l}_{i=1}
\left[\Delta_{\mathbb{H}}+\frac{(n-2i)(n+2i-2)}{4}\right]f
-\\
&2(l+2)x_{1}\frac{\partial}{\partial x_{1}}\prod^{l}_{i=1}
\left[\Delta_{\mathbb{H}}+\frac{(n-2i)(n+2i-2)}{4}\right]f.
\end{split}
\end{equation}
Combing (\ref{5.8}) and (\ref{5.9}) yields
\begin{equation}\label{5.10}
\begin{split}
&x^{\frac{n}{2}+l+1}_{1}\Delta^{l+1}(x^{l+1-\frac{n}{2}}_{1}f)\\
=&\left[\Delta_{\mathbb{H}}+\frac{(n+2l)(n+2l+2)}{4}\right]\prod^{l}_{i=1}
\left[\Delta_{\mathbb{H}}+\frac{(n-2i)(n+2i-2)}{4}\right]f
-\\
&2(l+2)x_{1}\frac{\partial}{\partial x_{1}}\prod^{l}_{i=1}
\left[\Delta_{\mathbb{H}}+\frac{(n-2i)(n+2i-2)}{4}\right]f+\\
&2(l+1)x^{\frac{n}{2}+l+1}_{1}\frac{\partial }{\partial x_{1}}\left\{x^{-\frac{n}{2}-l}_{1}\prod^{l}_{i=1}\left[\Delta_{\mathbb{H}}+\frac{(n-2i)(n+2i-2)}{4}\right]f\right\}\\
=&\left[\Delta_{\mathbb{H}}+\frac{(n+2l)(n+2l+2)}{4}\right]\prod^{l}_{i=1}
\left[\Delta_{\mathbb{H}}+\frac{(n-2i)(n+2i-2)}{4}\right]f-\\
&2(l+1)x^{1-\frac{n}{2}-l}_{1}\cdot\frac{\partial x^{\frac{n}{2}+l}_{1}}{\partial x_{1}}\cdot\prod^{l}_{i=1}
\left[\Delta_{\mathbb{H}}+\frac{(n-2i)(n+2i-2)}{4}\right]f.
\end{split}
\end{equation}
To get the last equation in the above, we use the fact
\begin{equation*}
\begin{split}
&x_{1}\frac{\partial}{\partial x_{1}}\prod^{l}_{i=1}
\left[\Delta_{\mathbb{H}}+\frac{(n-2i)(n+2i-2)}{4}\right]f
\\
=&x_{1}\frac{\partial}{\partial x_{1}}\left\{x^{\frac{n}{2}+l}_{1}x^{-\frac{n}{2}-l}_{1}\prod^{l}_{i=1}
\left[\Delta_{\mathbb{H}}+\frac{(n-2i)(n+2i-2)}{4}\right]f\right\}\\
=&
x^{\frac{n}{2}+l+1}_{1}\frac{\partial }{\partial x_{1}}\left\{x^{-\frac{n}{2}-l}_{1}\prod^{l}_{i=1}\left[\Delta_{\mathbb{H}}+\frac{(n-2i)(n+2i-2)}{4}\right]f\right\}+\\
&x^{1-\frac{n}{2}-l}_{1}\frac{\partial x^{\frac{n}{2}+l}_{1}}{\partial x_{1}}\cdot\prod^{l}_{i=1}
\left[\Delta_{\mathbb{H}}+\frac{(n-2i)(n+2i-2)}{4}\right]f.
\end{split}
\end{equation*}
Therefore, by (\ref{5.10}),
\begin{equation}\label{5.11}
\begin{split}
&x^{\frac{n}{2}+l+1}_{1}\Delta^{l+1}(x^{l+1-\frac{n}{2}}_{1}f)\\
=&\left[\Delta_{\mathbb{H}}+\frac{(n+2l)(n+2l+2)}{4}\right]\prod^{l}_{i=1}
\left[\Delta_{\mathbb{H}}+\frac{(n-2i)(n+2i-2)}{4}\right]f-\\
&2(l+1)\left(\frac{n}{2}+l\right)\prod^{l}_{i=1}
\left[\Delta_{\mathbb{H}}+\frac{(n-2i)(n+2i-2)}{4}\right]f\\
=&\prod^{l+1}_{i=1}
\left[\Delta_{\mathbb{H}}+\frac{(n-2i)(n+2i-2)}{4}\right]f.
\end{split}
\end{equation}
The proof of Lemma 5.1 is thus completed.
\end{proof}

\textbf{Proof of Theorem 1.3}
We claim that
\begin{equation}\label{1.11}
\int_{\mathbb{R}^{n}_{+}}|\nabla^{k}u|^{2}dx- \prod^{k}_{i=1}\frac{(2i-1)^{2}}{4}\int_{\mathbb{R}^{n}_{+}}\frac{u^{2}}{x^{2k}_{1}}dx=\int_{\mathbb{B}^{n}}(P_{k}v)vdV- \prod^{k}_{i=1}\frac{(2i-1)^{2}}{4}\int_{\mathbb{B}^{n}}v^{2}dV,
\end{equation}
where $v=x^{\frac{n}{2}-k}_{1}u$. In fact, by Lemma 5.1,
\begin{equation}
\begin{split}
&\int_{\mathbb{R}^{n}_{+}}|\nabla^{k}u|^{2}dx- \prod^{k}_{i=1}\frac{(2i-1)^{2}}{4}\int_{\mathbb{R}^{n}_{+}}\frac{u^{2}}{x^{2k}_{1}}dx\\
=&
\int_{\mathbb{R}^{n}_{+}}u\cdot (-\Delta)^{k}u dx- \prod^{k}_{i=1}\frac{(2i-1)^{2}}{4}\int_{\mathbb{R}^{n}_{+}}\frac{u^{2}}{x^{2k}_{1}}dx\\
=&\int_{\mathbb{H}^{n}}x^{\frac{n}{2}-k}_{1}u\cdot P_{k}(x^{\frac{n}{2}-k}_{1}u)dV- \prod^{k}_{i=1}\frac{(2i-1)^{2}}{4}\int_{\mathbb{H}^{n}_{+}}(x^{\frac{n}{2}-k}_{1}u)^{2}dV\\
=&\int_{\mathbb{B}^{n}}(P_{k}v)vdV- \prod^{k}_{i=1}\frac{(2i-1)^{2}}{4}\int_{\mathbb{B}^{n}}v^{2}dV.
\end{split}
\end{equation}
Therefore, by Corollary 1.3,
\begin{equation*}
\begin{split}
&\int_{\mathbb{R}^{n}_{+}}|\nabla^{k}u|^{2}dx- \prod^{k}_{i=1}\frac{(2i-1)^{2}}{4}\int_{\mathbb{R}^{n}_{+}}\frac{u^{2}}{x^{2k}_{1}}dx\\
=&\int_{\mathbb{B}^{n}}(P_{k}v)vdV- \prod^{k}_{i=1}\frac{(2i-1)^{2}}{4}\int_{\mathbb{B}^{n}}v^{2}dV\\
\geq& C\left(\int_{\mathbb{H}^{n}}\left|v\right|^{p}dV\right)^{\frac{2}{p}}=C\left(\int_{\mathbb{R}^{n}_{+}}
x^{\gamma}_{1}|u|^{p}dx\right)^{\frac{2}{p}}.
\end{split}
\end{equation*}

Similarly, using the identity (\ref{b7.1}), we obtain inequality (\ref{1.14}).
The proof of Theorem 1.4 is thereby completed.

\section{an alternative proof of  the work of Benguria,  Frank and   Loss in case $n=3$}
Firstly we recall the result of   Benguria,  Frank and   Loss (see \cite{bfl},  Corollary 3.1):
\begin{theorem}
Let $n\geq2$ and $n-1\leq \alpha<n$ (resp. $0<\alpha<1$ if $n=1$). The operator $(-\Delta-\frac{1}{4x^{2}_{1}})^{-\frac{\alpha}{2}}$
is a bounded operator from $L^{p}(\mathbb{R}^{n}_{+})$ to $L^{q}(\mathbb{R}^{n}_{+})$ for all $1<p,q<\infty$ that satisfy
$\frac{1}{q}=\frac{1}{p}-\frac{\alpha}{n},$  and its norm coincides with the one of $(-\Delta)^{-\frac{\alpha}{2}}:
L^{p}(\mathbb{R}^{n})\rightarrow L^{q}(\mathbb{R}^{n})$.

\medskip

Moreover, for such values of $\alpha$ we have
\begin{equation}\label{6.1}
\begin{split}
\left(f, (-\Delta-\frac{1}{4x^{2}_{1}})^{-\frac{\alpha}{2}}f\right)\leq 2^{-\alpha}\pi^{-\frac{n}{2}}\frac{\Gamma(\frac{n-\alpha}{2})}{\Gamma(\frac{\alpha}{2})}
C_{n,n-\alpha}\|f\|_{p},
\end{split}
\end{equation}
where $p=\frac{2n}{n+\alpha}$ and $C_{n,n-\alpha}$ is given by (\ref{1.10}).
Furthermore, the constant $C_{n,n-\alpha}$ is the sharp constant and this constant is not attained in (\ref{6.1}) for nonzero functions.
\end{theorem}

Choosing $n=3$ and $\alpha=2$ in Theorem 6.1 yields the following sharp Hardy-Sobolev-Maz'ya inequality (see \cite{bfl}, Theorem 1.1).
\begin{theorem}
Let $f\in D^{1}(\mathbb{R}^{3}_{+})$. Then $f\in L^{6}(\mathbb{R}^{3}_{+})$ and the inequality
\[
\int_{\mathbb{R}^{3}_{+}}|\nabla f|^{2}dx- \frac{1}{4}\int_{\mathbb{R}^{3}_{+}}\frac{| f|^{2}}{x^{2}_{1}}dx \geq S_{3}\left(\int_{\mathbb{R}^{3}_{+}}| f|^{6}dx\right)^{\frac{1}{3}}\]
holds,  where $S_{3}$ is the sharp Sobolev constant in three dimensions, i.e., $S_{3}=3(\pi/2)^{4/3}$.
The inequality is always strict for nonzero $f$'s. Using the formulation $g=\frac{f}{\sqrt{x_{1}}}$  we have the inequality
\begin{equation}\label{6.2}
\begin{split}
\int_{\mathbb{H}^{3}}|\nabla_{\mathbb{H}} g|^{2}dV- \int_{\mathbb{H}^{3}}| g|^{2}dV \geq S_{3}\left(\int_{\mathbb{H}^{3}}|g|^{6}dV\right)^{\frac{1}{3}}.
\end{split}
\end{equation}
\end{theorem}

In this section, we shall give an  alternative proof of (\ref{6.2}). In fact, we have the following
\begin{theorem}
Let $1\leq\alpha<3$ and $p=\frac{6}{3+\alpha}$. Then for $f\in L^{p}(\mathbb{B}^{3})$, we have
\begin{equation}\label{6.3}
\begin{split}
\int_{\mathbb{B}^{3}}f(x) \left[(-\Delta_{\mathbb{H}}-1)^{-\frac{\alpha}{2}}f\right](x)dV\leq 2^{-\alpha}\pi^{-\frac{3}{2}}\frac{\Gamma(\frac{3-\alpha}{2})}{\Gamma(\frac{\alpha}{2})} C_{3,3-\alpha}\left(\int_{\mathbb{B}^{3}}|f|^{\frac{6}{3+\alpha}}dV\right)^{\frac{3+\alpha}{3}},
\end{split}
\end{equation}
where  $C_{3,3-\alpha}$ is given by (\ref{1.10}).
Furthermore, the constant $C_{3,3-\alpha}$ is  sharp  and this constant is not attained in (\ref{6.3}) for nonzero functions.
\end{theorem}
\begin{proof}
We have, for $1\leq\alpha<3$,
\begin{equation*}
\begin{split}
(-\Delta_{\mathbb{H}}-1)^{-\frac{\alpha}{2}}=&\frac{1}{\Gamma(\frac{\alpha}{2})}\int^{\infty}_{0}t^{\frac{\alpha}{2}-1}e^{t(\Delta_{\mathbb{H}}+1)}dt\\
=&\frac{1}{\Gamma(\frac{\alpha}{2})}\cdot2^{-3}\pi^{-3/2}\frac{\rho}{\sinh \rho}\int^{\infty}_{0}t^{\frac{\alpha-5}{2}} e^{-\frac{\rho^{2}}{4t}}dt\\
=&2^{-\alpha}\pi^{-3/2}\frac{\Gamma(\frac{3-\alpha}{2})}{\Gamma(\frac{\alpha}{2})}\frac{1}{\rho^{2-\alpha}\sinh\rho}\\
=&2^{-\alpha}\pi^{-3/2}\frac{\Gamma(\frac{3-\alpha}{2})}{\Gamma(\frac{\alpha}{2})}\left(\frac{1}{2\sinh\frac{\rho}{2}}\right)^{3-\alpha}
\cdot\Psi_{\alpha}(\rho),\\
\end{split}
\end{equation*}
where $\Psi_{\alpha}(\rho)=\left(\frac{2\sinh\frac{\rho}{2}}{\rho}\right)^{2-\alpha}\frac{1}{\cosh\frac{\rho}{2}}$.
It is easy to check, for $\alpha\geq 1$, the function $\Psi_{\alpha}(\rho)$ is decrease on $(0,\infty)$ and
\[
\sup_{\rho\in(0,\infty)}\Psi_{\alpha}(\rho)=1.
\]
Therefore, By Theorem 1.2,
\begin{equation*}
\begin{split}
\int_{\mathbb{B}^{3}}f(x) \left[(-\Delta_{\mathbb{H}}-1)^{-\frac{\alpha}{2}}f\right](x)dV\leq& 2^{-\alpha}\pi^{-3/2}\frac{\Gamma(\frac{3-\alpha}{2})}{\Gamma(\frac{\alpha}{2})}
\int_{\mathbb{B}^{3}}\int_{\mathbb{B}^{3}}\frac{f(x)f(y)}{\left(2\sinh\frac{\rho(T_{y}(x))}{2}\right)^{3-\alpha}}dV_{x}dV_{y}\\
\leq& 2^{-\alpha}\pi^{-\frac{3}{2}}\frac{\Gamma(\frac{3-\alpha}{2})}{\Gamma(\frac{\alpha}{2})}C_{3,3-\alpha}
\left(\int_{\mathbb{B}^{3}}|f|^{\frac{6}{3+\alpha}}dV\right)^{\frac{3+\alpha}{3}}.
\end{split}
\end{equation*}
Also by Theorem 1.2, the constant $C_{3,3-\alpha}$ is sharp and not attained  for nonzero functions.
\end{proof}

\begin{remark}
 Choosing $\alpha=2$ in Theorem 6.3 and closely following the proof of Theorem 1.1 in \cite{bfl}, we get (\ref{6.2}).
\end{remark}

\begin{remark}
Since choosing $n=3$ and $\alpha=2$ in (\ref{6.1}) yields a sharp Hardy-Sobolev-Maz'ya inequality on $\mathbb{R}^{3}_{+}$,
  a natural question is whether inequality (\ref{6.1}) implies some Hardy-Sobolev-Maz'ya inequality for  higher order derivatives on half spaces.
  However, it seems that this is not   the case. For example, choosing $n=5$ and $\alpha=4$ in (\ref{6.1}) and following the proof in (\cite{bfl}), we have the following sharp Sobolev type inequality
\begin{equation}\label{6.5}
\begin{split}
\int_{\mathbb{R}^{5}_{+}}\left|\Delta u+\frac{u}{4x^{2}_{1}}\right|^{2}dx\geq S_{5,2}\left(\int_{\mathbb{R}^{5}_{+}}|u|^{10}dx\right)^{\frac{1}{5}},
\end{split}
\end{equation}
where  $S_{5,2}$ is the best $2$-th order Sobolev constant on $\mathbb{R}^{5}$. Now we compute the left of inequality (\ref{6.5}).
\end{remark}

\begin{lemma}
Let $u\in C^{\infty}_{0}(\mathbb{R}^{5}_{+})$ and $f=x^{\frac{1}{2}}_{1}u$. Then
\begin{equation*}
\begin{split}
\int_{\mathbb{R}^{5}_{+}}\left|\Delta u+\frac{u}{4x^{2}_{1}}\right|^{2}dx=\int_{\mathbb{H}^{5}}|\Delta_{\mathbb{H}}f+3f|^{2}
dV.
\end{split}
\end{equation*}

\end{lemma}
\begin{proof} We compute
\begin{equation}\label{6.6}
\begin{split}
\int_{\mathbb{R}^{5}_{+}}\left|\Delta u+\frac{u}{4x^{2}_{1}}\right|^{2}dx=&\int_{\mathbb{R}^{5}_{+}}\left|-\Delta u-\frac{u}{4x^{2}_{1}}\right|^{2}dx\\
=&\int_{\mathbb{R}^{5}_{+}}\left(u(-\Delta)^{2} u
-\frac{u(-\Delta)u}{2x^{2}_{1}}+\frac{u^{2}}{16x^{4}_{1}}\right)dx.
\end{split}
\end{equation}

Since $u=x^{-\frac{1}{2}}_{1}f$, we have, by (\ref{5.3}) and (\ref{5.1}),
\begin{equation*}
\begin{split}
(-\Delta)u=&(-\Delta)(x^{-\frac{1}{2}}_{1}f)=x^{-\frac{5}{2}}_{1}\left(-\Delta_{\mathbb{H}}-\frac{3}{4}\right)
f-2x^{-\frac{3}{2}}_{1}\frac{\partial f}{\partial x_{1}};\\
(-\Delta)^{2}u=&(-\Delta)^{2}(x^{-\frac{1}{2}}_{1}f)=x^{-\frac{9}{2}}_{1}\left(-\Delta_{\mathbb{H}}-\frac{15}{4}\right)
\left(-\Delta_{\mathbb{H}}-\frac{7}{4}\right)f.
\end{split}
\end{equation*}
Therefore,
\begin{equation}\label{6.7}
\begin{split}
\int_{\mathbb{R}^{5}_{+}}u(-\Delta)^{2} u=&\int_{\mathbb{R}^{5}_{+}}x^{-5}_{1}f\left(-\Delta_{\mathbb{H}}-\frac{15}{4}\right)
\left(-\Delta_{\mathbb{H}}-\frac{7}{4}\right)fdx\\
=&\int_{\mathbb{H}^{5}}f\left(-\Delta_{\mathbb{H}}-\frac{15}{4}\right)
\left(-\Delta_{\mathbb{H}}-\frac{7}{4}\right)fdV;\\
-\int_{\mathbb{R}^{5}_{+}}
\frac{u(-\Delta)u}{2x^{2}_{1}}dx=&-\frac{1}{2}\int_{\mathbb{R}^{5}_{+}}x^{-5}_{1}f\left(-\Delta_{\mathbb{H}}-\frac{3}{4}\right)
fdx+\int_{\mathbb{R}^{5}_{+}}x^{-4}_{1}f\frac{\partial f}{\partial x_{1}}dx\\
=&-\frac{1}{2}\int_{\mathbb{H}^{5}}f\left(-\Delta_{\mathbb{H}}-\frac{3}{4}\right)
fdV+2\int_{\mathbb{R}^{5}_{+}}x^{-5}_{1}f^{2}dx\\
=&-\frac{1}{2}\int_{\mathbb{H}^{5}}f\left(-\Delta_{\mathbb{H}}-\frac{3}{4}\right)
fdV+2\int_{\mathbb{H}^{5}}f^{2}dV;\\
\int_{\mathbb{R}^{5}_{+}}\frac{u^{2}}{16x^{4}_{1}}dx=&\int_{\mathbb{R}^{5}_{+}}\frac{f^{2}}{16x^{5}_{1}}dx=\frac{1}{16}\int_{\mathbb{H}^{5}}f^{2}dV.
\end{split}
\end{equation}
Combing (\ref{6.6}) and (\ref{6.7}) yields
\begin{equation*}
\begin{split}
&\int_{\mathbb{R}^{5}_{+}}\left|-\Delta u-\frac{u}{4x^{2}_{1}}\right|^{2}dx\\
=&\int_{\mathbb{H}^{5}}f\left(-\Delta_{\mathbb{H}}-\frac{15}{4}\right)
\left(-\Delta_{\mathbb{H}}-\frac{7}{4}\right)fdV-
\frac{1}{2}\int_{\mathbb{H}^{5}}f\left(-\Delta_{\mathbb{H}}-\frac{3}{4}\right)
fdV+\frac{33}{16}\int_{\mathbb{H}^{5}}f^{2}dx\\
=&\int_{\mathbb{H}^{5}}f\left(-\Delta_{\mathbb{H}}-3\right)^{2}
fdV\\
=&\int_{\mathbb{H}^{5}}|\Delta_{\mathbb{H}}f+3f|^{2}
dV.
\end{split}
\end{equation*}
This completes the proof.
\end{proof}

By Lemma 6.2 and  the Plancherel formula (\ref{2.5}),
\begin{equation*}
\begin{split}
\int_{\mathbb{R}^{5}_{+}}\left|\Delta u+\frac{u}{4x^{2}_{1}}\right|^{2}dx=&D_{5}\int^{+\infty}_{-\infty}\int_{\mathbb{S}^{4}}\frac{(\lambda^{2}+4)^{2}}{16}|\widehat{f}(\lambda,\zeta)|^{2}
|\mathfrak{c}(\lambda)|^{-2}d\lambda d\sigma(\varsigma).
\end{split}
\end{equation*}
However,
\begin{equation*}
\begin{split}
&\int_{\mathbb{H}^{5}}(P_{4}f)fdV- \prod^{2}_{i=1}\frac{(2i-1)^{2}}{4}\int_{\mathbb{H}^{5}}f^{2}dV\\
=&
D_{5}\int^{+\infty}_{-\infty}\int_{\mathbb{S}^{4}}\prod^{2}_{i=1}\frac{\lambda^{2}+(2i-1)^{2}}{4}|\widehat{f}(\lambda,\zeta)|^{2}|\mathfrak{c}(\lambda)|^{-2}d\lambda d\sigma(\varsigma)-\\
&D_{5}\prod^{2}_{i=1}\frac{(2i-1)^{2}}{4}\int^{+\infty}_{-\infty}\int_{\mathbb{S}^{4}}|\widehat{f}(\lambda,\zeta)|^{2}|\mathfrak{c}(\lambda)|^{-2}d\lambda d\sigma(\varsigma)\\
=&D_{5}\int^{+\infty}_{-\infty}\int_{\mathbb{S}^{4}}\frac{\lambda^{4}+10\lambda^{2}}{16}|\widehat{f}(\lambda,\zeta)|^{2}|\mathfrak{c}(\lambda)|^{-2}d\lambda d\sigma(\varsigma).
\end{split}
\end{equation*}
Since
\[
\inf_{\lambda\in\mathbb{R}}\frac{\frac{\lambda^{4}+10\lambda^{2}}{16}}
{\frac{(\lambda^{2}+4)^{2}}{16}}=0,
\]
one cannot find a positive constant C which is independent of $f$ such that
\begin{equation*}
\begin{split}
\int_{\mathbb{H}^{n}}(P_{4}f)fdV- \prod^{k}_{i=1}\frac{(2i-1)^{2}}{4}\int_{\mathbb{H}^{n}}f^{2}dV\geq C\int_{\mathbb{R}^{5}_{+}}\left|\Delta u+\frac{u}{4x^{2}_{1}}\right|^{2}dx,
\end{split}
\end{equation*}
where $u=x^{-\frac{1}{2}}_{1}f$.
Therefore, inequality (\ref{6.5}) does not imply the following Hardy-Sobolev-Maz'ya inequality on $\mathbb{R}^{5}_{+}$
\[
\int_{\mathbb{R}^{5}_{+}}|\Delta u|^{2}dx- \frac{9}{16}\int_{\mathbb{R}^{5}_{+}}\frac{u^{2}}{x^{4}_{1}}dx\geq C\left(\int_{\mathbb{R}^{n}_{+}}|u|^{10}dx\right)^{\frac{1}{5}}.
\]

\section{an alternative proof of the sharp Sobolev inequalities  of G. Liu}

In this section, we shall give  an alternative proof of the work of G. Liu concerning the sharp constant in the Sobolev inequality in hyperbolic spaces via
Hardy-Littlewood-Sobolev inequalities \cite{liu}.   Before we begin the proof, we need the following Lemma.

\begin{lemma}
Let $f\in C^{\infty}_{0}(\mathbb{B}^{n})$. Then there exists a function $g\in C^{\infty}_{0}(\mathbb{B}^{n})$  such that $P_{k}g=f$.
\end{lemma}
\begin{proof}
Choose $\widetilde{f}\in  C^{\infty}_{0}(\mathbb{B}^{n})$ such that
$$(-\Delta)^{k}\widetilde{f}=\left(\frac{1-|x|^{2}}{2}\right)^{-k-\frac{n}{2}}f.$$
Set $g=(\frac{1-|x|^{2}}{2})^{k-\frac{n}{2}}\widetilde{f}$. Then $g\in C^{\infty}_{0}(\mathbb{B}^{n})$. Furthermore,
 by (\ref{b7.1}),
\[
P_{k}g=P_{k}\left[\left(\frac{1-|x|^{2}}{2}\right)^{k-\frac{n}{2}}\widetilde{f}\right]=\left(\frac{1-|x|^{2}}{2}\right)^{k+\frac{n}{2}}(-\Delta)^{k}
\widetilde{f}=f.
\]
\end{proof}

It is known that the  kernel $(-\Delta)^{-k}(1\leq k<n/2)$ is $\frac{1}{\gamma(2k)|x|^{n-2k}}$, where
\begin{equation}\label{7.1}
\gamma(2k)=\frac{\pi^{n/2}2^{2k}\Gamma(k)}{\Gamma(\frac{n}{2}-k)}.
\end{equation}
We have, for $f\in C^{\infty}_{0}(\mathbb{B}^{n})$,
\begin{equation}\label{7.2}
f(0)=\frac{1}{\gamma(2k)}\int_{\mathbb{B}^{n}}(-\Delta)^{k}f(x)\cdot\frac{1}{|x|^{n-2k}}dx.
\end{equation}
Replacing $f$ by $(1-|x|^{2})^{k-\frac{n}{2}}f$ and using (\ref{b7.1}), we obtain
\begin{equation}\label{7.3}
\begin{split}
f(0)=&\frac{1}{\gamma(2k)}\int_{\mathbb{B}^{n}}(-\Delta)^{k}[(1-|x|^{2})^{k-\frac{n}{2}}f]\cdot\frac{1}{|x|^{n-2k}}dx\\
=&\frac{2^{k-\frac{n}{2}}}{\gamma(2k)}\int_{\mathbb{B}^{n}}P_{k}f(x)\cdot\frac{1}{|x|^{n-2k}}\left(\frac{1-|x|^{2}}{2}\right)^{-\frac{n}{2}-k}dx\\
=&\frac{1}{\gamma(2k)}\int_{\mathbb{B}^{n}}P_{k}f(x)\cdot\left(\frac{\sqrt{1-|x|^{2}}}{2|x|}\right)^{n-2k}dV\\
=&\frac{1}{\gamma(2k)}\int_{\mathbb{B}^{n}}P_{k}f(x)\cdot\left(2\sinh\frac{\rho(x)}{2}\right)^{2k-n}dV.
\end{split}
\end{equation}
Therefore, by the M\"obius shift invariance, for $f\in C^{\infty}_{0}(\mathbb{B}^{n})$ and $y\in\mathbb{B}^{n}$,
\begin{equation}\label{7.4}
\begin{split}
f(y)=&\frac{1}{\gamma(2k)}\int_{\mathbb{B}^{n}}P_{k}f(x)\cdot\left(2\sinh\frac{\rho(T_{y}(x))}{2}\right)^{2k-n}dV_{x}.
\end{split}
\end{equation}

\begin{theorem}
Let $1\leq k<n/2$. Then, for any $f\in C^{\infty}_{0}(\mathbb{B}^{n})$,
\[
\int_{\mathbb{B}^{n}}P_{k}f(x)\cdot f(x)dV\geq S_{n,k}\left(\int_{\mathbb{B}^{n}}|f(x)|^{\frac{2n}{n-2k}}dV\right)^{\frac{n-2k}{n}},
\]
where $S_{n,k}$ is the best $k$-th order Sobolev constant. Furthermore, the inequality is strict for nonzero $f$'s.
\end{theorem}
\begin{proof}
Let $g\in C^{\infty}_{0}(\mathbb{B}^{n})$ be such that $P_{k}g=f$. Then
\begin{equation}\label{7.5}
\begin{split}
\int_{\mathbb{B}^{n}}f(x)g(x)dV=&\int_{\mathbb{B}^{n}}(P^{-1}_{k}f)(x)\cdot(P_{k}g)(x)dV\\
=&\int_{\mathbb{B}^{n}}(P^{-1}_{k}f)(x)\cdot f(x)dV\\
=&\int_{\mathbb{B}^{n}}|(P^{-1/2}_{k}f)(x)|^{2}dV.
\end{split}
\end{equation}
By (\ref{7.4}) and Theorem 1.2,
\begin{equation}\label{7.6}
\begin{split}
\int_{\mathbb{B}^{n}}f(x)g(x)dV=&\frac{1}{\gamma(2k)}\int_{\mathbb{B}^{n}}\int_{\mathbb{B}^{n}}\frac{f(x)\cdot P_{k}g}{(2\sinh\frac{\rho(T_{y}(x))}{2})^{n-2k}}dV_{x}dV_{y}\\
=&\frac{1}{\gamma(2k)}\int_{\mathbb{B}^{n}}\int_{\mathbb{B}^{n}}\frac{f(x)\cdot f(y)}{(2\sinh\frac{\rho(T_{y}(x))}{2})^{n-2k}}dV_{x}dV_{y}\\
\leq& \frac{1}{\gamma(2k)} C_{n,n-2k}\left(\int_{\mathbb{B}^{n}}|f(x)|^{\frac{2n}{n+2k}}dV\right)^{\frac{n+2k}{n}}.
\end{split}
\end{equation}
Combing (7.5) and (\ref{7.6}) yields
\begin{equation}\label{7.7}
\begin{split}
\int_{\mathbb{B}^{n}}|(P^{-1/2}_{k}f)(x)|^{2}dV
\leq& \frac{1}{\gamma(2k)} C_{n,n-2k}\left(\int_{\mathbb{B}^{n}}|f(x)|^{\frac{2n}{n+2k}}dV\right)^{\frac{n+2k}{n}},\;\; f\in C^{\infty}_{0}(\mathbb{B}^{n}).
\end{split}
\end{equation}

On the other hand, for $h\in C^{\infty}_{0}(\mathbb{B}^{n})$,
\begin{equation}\label{7.8}
\begin{split}
\left|\int_{\mathbb{B}^{n}}h(x)f(x)dV\right|^{2}=&\left|\int_{\mathbb{B}^{n}}(P^{\frac{1}{2}}_{k}h)(x)(P^{-\frac{1}{2}}_{k}f)(x)dV\right|^{2}\\
\leq&\int_{\mathbb{B}^{n}}|(P^{\frac{1}{2}}_{k}h)(x)|^{2}dV\cdot\int_{\mathbb{B}^{n}}|(P^{-\frac{1}{2}}_{k}f)(x)|^{2}dV.
\end{split}
 \end{equation}
Combing (\ref{7.7}) and (\ref{7.8}) yields
\begin{equation}\label{7.9}
\begin{split}
\left|\int_{\mathbb{B}^{n}}h(x)f(x)dV\right|^{2}\leq&\frac{1}{\gamma(2k)} C_{n,n-2k}\int_{\mathbb{B}^{n}}|(P^{\frac{1}{2}}_{k}h)(x)|^{2}dV
\left(\int_{\mathbb{B}^{n}}|f(x)|^{\frac{2n}{n+2k}}dV\right)^{\frac{n+2k}{n}}\\
=&\frac{1}{\gamma(2k)} C_{n,n-2k}\int_{\mathbb{B}^{n}}P_{k}h(x)\cdot h(x)dV
\left(\int_{\mathbb{B}^{n}}|f(x)|^{\frac{2n}{n+2k}}dV\right)^{\frac{n+2k}{n}}.
\end{split}
 \end{equation}
Taking $f=h^{\frac{n+2k}{n-2k}}$, we have, by (\ref{7.9}),
\begin{equation}
\begin{split}
\left(\int_{\mathbb{B}^{n}}|h(x)|^{\frac{2n}{n-2k}}dV\right)^{2}\leq&\frac{1}{\gamma(2k)} C_{n,n-2k}\int_{\mathbb{B}^{n}}P_{k}h(x)\cdot h(x)dV
\left(\int_{\mathbb{B}^{n}}|h(x)|^{\frac{2n}{n-2k}}dV\right)^{\frac{n+2k}{n}}.\\
\end{split}
 \end{equation}
Therefore,
\begin{equation}
\begin{split}
\frac{\gamma(2k)}{C_{n,n-2k}}\left(\int_{\mathbb{B}^{n}}|h(x)|^{\frac{2n}{n-2k}}dV\right)^{\frac{n-2k}{n}}\leq\int_{\mathbb{B}^{n}}P_{k}h(x)\cdot h(x)dV,
\end{split}
 \end{equation}
where
\[
\frac{\gamma(2k)}{C_{n,n-2k}}=\frac{\pi^{n/2}2^{2k}\Gamma(k)}{\Gamma(n/2-k)}\cdot
\frac{\Gamma(n/2+k)}{\pi^{n/2-k}\Gamma(k)}\left(\frac{\Gamma(n/2)}{\Gamma(n)}\right)^{2k/n}=2^{2k}\pi^{k}\frac{\Gamma(n/2+k)}{n/2-k}
\left(\frac{\Gamma(n/2)}{\Gamma(n)}\right)^{2k/n}
\]
is the best $k$-th order Sobolev constant (see e.g. \cite{ct}). Moreover, by Theorem 1.2,  the inequality is strict for nonzero $f$'s.
\end{proof}

\section{sharp constants in the Sobolev inequality and the Hardy-Sobolev-Maz'ya inequality are the same in case $n=5$ and $k=2$}

In this section we shall show that the sharp constant of Hardy-Sobolev-Maz'ya inequality  for $n=5$ and $k=2$ coincides with
the corresponding Sobolev constant. The proof depends on the following key lemma.
\begin{lemma}
Let $n=5$. There holds
\[
[(-\Delta_{\mathbb{H}}-4)(-\Delta_{\mathbb{H}}-3)]^{-1}=\frac{1}{16\pi^{2}}\cdot\frac{1}{2\sinh\frac{\rho}{2}}\cdot\frac{1}{\cosh^{2}\frac{\rho}{2}}\leq\frac{1}{\gamma(4)}
\cdot\frac{1}{2\sinh\frac{\rho}{2}},\;\;\rho>0,
\]
where $\gamma(4)=\frac{1}{16\pi^{2}}$ is defined in (\ref{7.1}).
\end{lemma}
\begin{proof}
By (\ref{3.8}), we have, for $n=5$,
\begin{equation*}
\begin{split}
\left(-\Delta_{\mathbb{H}}-4\right)^{-1}=&2^{-3}\pi^{-2}\left(-\frac{1}{\sinh \rho}\frac{\partial}{\partial \rho}\right)\frac{1}{\sinh \rho}=\frac{1}{8\pi^{2}}\cdot\frac{\cosh\rho}{\sinh^{3}\rho}.
\end{split}
\end{equation*}
On the other hand, choosing $\lambda=-3$ and $\theta_{5}(\lambda)=\frac{1}{2}$ in (\ref{3.5}), we have
\begin{equation*}
\begin{split}
(-\Delta_{\mathbb{H}}-3)^{-1}
=&\frac{1}{(\sinh\rho)^{3}}\cdot(2\pi)^{-\frac{5}{2}}\frac{\Gamma(\frac{5}{2}+\frac{1}{2})}{2^{\frac{1}{2}+1}
\Gamma(\frac{1}{2}+1)}\int^{\pi}_{0}\sin^{2} tdt\\
=&\frac{1}{8\pi^{2}}\cdot\frac{1}{\sinh^{3}\rho}.
\end{split}
\end{equation*}
Therefore,
\begin{equation*}
\begin{split}
[(-\Delta_{\mathbb{H}}-4)(-\Delta_{\mathbb{H}}-3)]^{-1}=&(-\Delta_{\mathbb{H}}-4)^{-1}-(-\Delta_{\mathbb{H}}-3)^{-1}\\
=&\frac{1}{8\pi^{2}}\cdot\frac{\cosh\rho}{\sinh^{3}\rho}-\frac{1}{8\pi^{2}}\cdot\frac{1}{\sinh^{3}\rho}\\
=&\frac{1}{16\pi^{2}}\cdot\frac{1}{2\sinh\frac{\rho}{2}}\cdot\frac{1}{\cosh^{2}\frac{\rho}{2}}
\end{split}
\end{equation*}
The desired result follows.
\end{proof}

By Theorem 1.2 and closely following the proof of Theorem 7.2, we have
\begin{corollary}
There holds,  for each $u\in C^{\infty}_{0}(\mathbb{B}^{n})$,
\begin{equation}\label{1.11}
\int_{\mathbb{B}^{5}}((-\Delta_{\mathbb{H}}-4)(-\Delta_{\mathbb{H}}-3)u)udV\geq S_{5,2}\left(\int_{\mathbb{B}^{n}}|u|^{\frac{10}{3}}dV\right)^{\frac{3}{5}},
\end{equation}
where $S_{5,2}=\gamma(4)/C_{5,1}$ is the best second order Sobolev constant.
\end{corollary}

\begin{theorem}
There holds,  for each $u\in C^{\infty}_{0}(\mathbb{B}^{n})$,
\begin{equation*}
\int_{\mathbb{B}^{5}}(P_{2}u)udV-\frac{9}{16}\int_{\mathbb{B}^{5}}u^{2}dV\geq S_{5,2}\left(\int_{\mathbb{B}^{n}}|u|^{\frac{10}{3}}dV\right)^{\frac{3}{5}}.
\end{equation*}
In term of half space model $\mathbb{H}^{n}$ and ball model $\mathbb{B}^{n}$, respectively, inequality above is  equivalent to the follows:
\begin{equation*}
\int_{\mathbb{B}^{5}}|\Delta f|^{2}dx- 9\int_{\mathbb{B}^{5}}\frac{f^{2}}{(1-|x|^{2})^{4}}dx\geq S_{5,2}\left(\int_{\mathbb{B}^{5}}|f|^{\frac{10}{3}}dx\right)^{\frac{3}{5}},\;\;f\in C^{\infty}_{0}(\mathbb{B}^{5});
\end{equation*}
\begin{equation*}
\int_{\mathbb{R}_{+}^{5}}|\Delta g|^{2}dx- \frac{9}{16}\int_{\mathbb{R}_{+}^{5}}\frac{g^{2}}{x_{1}^{4}}dx\geq S_{5,2}\left(\int_{\mathbb{R}_{+}^{5}}|g|^{\frac{10}{3}}dx\right)^{\frac{3}{5}},\;\;g\in C^{\infty}_{0}(\mathbb{R}_{+}^{5}).
\end{equation*}
\end{theorem}
\begin{proof}
By Corollary 8.2,  it is enough to show
\[
\int_{\mathbb{B}^{5}}(P_{2}u)udV-\frac{9}{16}\int_{\mathbb{B}^{5}}u^{2}dV\geq
\int_{\mathbb{B}^{5}}((-\Delta_{\mathbb{H}}-4)(-\Delta_{\mathbb{H}}-3)u)udV.
\]
In fact, by the Plancherel formula (see (\ref{2.5})),
\begin{equation*}
\begin{split}
&\int_{\mathbb{B}^{5}}(P_{2}u)udV-\frac{9}{16}\int_{\mathbb{B}^{5}}u^{2}dV\\
=&
D_{5}\int^{+\infty}_{-\infty}\int_{\mathbb{S}^{4}}\left[\frac{(\lambda^{2}+1)(\lambda^{2}+9)}{16}-\frac{9}{16}\right]|\widehat{u}(\lambda,\zeta)|^{2}|\mathfrak{c}(\lambda)|^{-2}d\lambda d\sigma(\varsigma))\\
=&D_{5}\int^{+\infty}_{-\infty}\int_{\mathbb{S}^{4}}\frac{\lambda^{4}+10\lambda^{2}}{16}|\widehat{u}(\lambda,\zeta)|^{2}|\mathfrak{c}(\lambda)|^{-2}d\lambda d\sigma(\varsigma))\\
\geq&D_{5}\int^{+\infty}_{-\infty}\int_{\mathbb{S}^{4}}\frac{\lambda^{2}(\lambda^{2}+1)}{16}|\widehat{u}(\lambda,\zeta)|^{2}|\mathfrak{c}(\lambda)|^{-2}d\lambda d\sigma(\varsigma))\\
=&\int_{\mathbb{B}^{5}}((-\Delta_{\mathbb{H}}-4)(-\Delta_{\mathbb{H}}-3)u)udV.
\end{split}
\end{equation*}
This completes the proof.
\end{proof}


\end{document}